\documentclass[11pt]{amsart}
\usepackage{amssymb,amsmath,amsthm}
\usepackage{verbatim}
\usepackage{enumerate}
\usepackage{graphicx}
\usepackage{epsfig}
\usepackage{color, soul}
\usepackage[all]{xy}
\usepackage{wasysym}
\usepackage{url}
\usepackage{mathrsfs}
\usepackage{cancel}

\DeclareMathOperator{\cb}{cb}

\DeclareMathOperator{\CB}{\mathcal{CB}}
\DeclareMathOperator{\tr}{tr}
\DeclareMathOperator{\proj}{proj}
\DeclareMathOperator{\spa}{span}

\DeclareMathOperator{\MIN}{MIN}
\DeclareMathOperator{\MAX}{MAX}
\DeclareMathOperator{\OH}{OH}

\DeclareMathOperator{\vN}{vN}

\newcommand{\tcc}{\tau_{\mathrm{cc}}}

\newcommand{\B}{ \mathcal{B} }

\newcommand{\n}[1]{ \left\|#1\right\| }

\newcommand{\N}{{\mathbb{N}}}

\newcommand{\C}{{\mathbb{C}}}

\newcommand{\pair}[2]{{\langle #1, #2 \rangle}}
\newcommand{\mpair}[2]{{\langle\langle #1, #2 \rangle\rangle}}

\newcommand\restr[2]{{
  \left.\kern-\nulldelimiterspace 
  #1 
  \vphantom{\big|} 
  \right|_{#2} 
  }}

\newcommand{\ORNP}[1]{{$#1$-tensor surjectivity property}}

\newtheorem{theorem}{Theorem}[section]
\newtheorem{lemma}[theorem]{Lemma}
\newtheorem{definition}[theorem]{Definition}
\newtheorem{corollary}[theorem]{Corollary}
\newtheorem{proposition}[theorem]{Proposition}
\newtheorem{remark}[theorem]{Remark}

\numberwithin{equation}{section}

\def\ker{{\rm ker\, }}

\begin{document}

\title{The $p$-Operator Approximation Property}

\author[J.A. Ch\'avez-Dom\'inguez]{Javier Alejandro Ch\'avez-Dom\'inguez}
\address{Department of Mathematics, University of Oklahoma, Norman, OK 73019-3103,
USA} \email{jachavezd@ou.edu}

\author[V. Dimant]{Ver\'onica Dimant}
\address{Departamento de Matem\'{a}tica y Ciencias, Universidad de San
Andr\'{e}s, Vito Dumas 284, (B1644BID) Victoria, Buenos Aires,
Argentina and CONICET} \email{vero@udesa.edu.ar}

\author[D. Galicer]{Daniel Galicer}
\address{Departamento de Matem\'{a}ticas y Estad\'{\i}stica, Universidad T. Di Tella, Av. Figueroa Alcorta 7350 (1428), Buenos Aires, Argentina and CONICET.
On leave from Departamento de Matem\'{a}tica, Facultad de Ciencias Exactas y Naturales, Universidad de Buenos Aires, (1428) Buenos Aires,
Argentina} \email{daniel.galicer@utdt.edu}
\date{}

\thanks{The first-named author was partially supported by NSF grants DMS-1900985 and DMS-2247374. The second-named author was partially supported by CONICET PIP 11220200101609CO and ANPCyT PICT 2018-04104. The third-named author was partially supported by CONICET-PIP 11220200102366CO
and ANPCyT PICT 2018-4250.}

\maketitle

\begin{abstract}
We study a notion analogous to the $p$-Approximation Property ($p$-AP) for Banach spaces, within the noncommutative context of operator spaces. Referred to as the $p$-Operator Approximation Property ($p$-OAP), this concept is linked to the ideal of operator $p$-compact mappings. We present several equivalent characterizations based on the density of finite-rank mappings within specific spaces for different topologies, and also one in terms of a slice mapping property.
Additionally, we investigate how this property transfers from the dual or bidual to the original space.
As an application, the $p$-OAP for the reduced $C^*$-algebra of a discrete group implies that operator $p$-compact Herz-Schur multipliers can be approximated in $\cb$-norm by finitely supported multipliers.
\end{abstract}

\section{Introduction}

The approximation property, prominently featured in the ``Scottish Book'', has significantly impacted modern mathematical theory.  A landmark moment came in the seventies when Enflo \cite{enflo1973counterexample} provided an example of a Banach space without the approximation property, solving a long-standing open question. 

Many variations of the approximation property have been considered in the literature throughout the years, the interested reader can find a nice account in the survey \cite{MR1863695}.
In many cases, these variations of the AP are \emph{stronger} properties where one requires additional conditions for the finite-rank maps used to approximate the identity. Some common examples of such conditions are being contractive, uniformly bounded, or commuting.

To give a taste of the fundamental ideas, recall that relatively compact subsets of a metric space are very useful and easy to handle because, up to any small error $\varepsilon>0$, they can be approximated by finite sets.
The approximation property (AP) of a Banach space $X$ is a functional analytic version of this idea of approximating by ``finite pieces''. This is already present in the most standard definition, i.e. that the identity map on $X$ can by approximated  by finite-rank operators uniformly on compact sets,
and there is an equivalent characterization where the analogy is arguably more transparent: a Banach space $X$ has AP if and only if for any Banach space $Y$, any compact operator $T : Y \to X$ can be approximated in operator norm by finite-rank operators.
There are many other equivalent characterizations of the AP, including the density of finite-rank mappings in $(\mathcal{L}(Y,X), \tau_c)$, the space of continuous linear operators equipped with the topology of uniform convergence on compact sets.

Inspired by Grothendieck's characterization of relatively compact sets \cite{grothendieck1955produits}, namely those enclosed within the convex hull of a null sequence of vectors, Sinha and Karn \cite{Sinha-Karn} introduced the stronger notion of $p$-compact sets. Loosely speaking, these sets correspond to Grothendieck's characterization above but replacing the null sequences in a Banach space by the smaller collection of the norm $p$-summable sequences.  This conceptual leap naturally gave rise to the notions of $p$-compact mappings and the $p$-approximation property ($p$-AP), which can be seen as a nuanced variation of the classical approximation property that focuses on the $p$-compact sets.
Since there are \emph{fewer} $p$-compact sets than compact ones,  
the $p$-approximation property is \emph{weaker} than the standard approximation property, in marked constrast with the aforementioned other variations of the AP. As Pietsch aptly stated in his book \cite{pietsch2007history}, \textit{``Life in Banach spaces with certain approximation property is much easier''}, so it is therefore intriguing to study this type of weaker properties.
For example, let us describe an interesting give-and-take that happens in this setting. As recalled above, the standard AP corresponds to being able to approximate, in operator norm, any compact mapping using finite-rank ones.
The $p$-approximation property, being weaker, is enjoyed by more spaces but yields a weaker conclusion: now we can only guarantee that the $p$-compact mappings can be approximated in operator norm by finite-rank ones.
Therefore, if one is working with a space which does not have the classical AP but it does have the $p$-AP, not everything is lost: the standard very useful machinery of approximating with finite-rank mappings is still available, just paying the price of having a smaller family of operators which can be approximated.

In the operator space setting, the history of approximation properties is similar the Banach space situation.
The most basic of such properties, corresponding to Grothendieck's AP, is the operator approximation property (OAP). The OAP is intimately connected to the operator space structure, and, as Pisier points out in his book \cite[Remark, page 85]{Pisier-Asterisque}, having the AP is entirely irrelevant when it comes to this stronger type of approximation. Indeed, there is a nice construction of Arias \cite{Arias} of an operator space isometric to a separable Hilbert space but lacking the OAP.

A variety of strengthenings of the OAP such as the strong operator approximation property (SOAP) and the completely bounded approximation property (CBAP) are standard in the theory.
Even stronger notions of approximation such as nuclearity are very useful in operator algebras, see e.g. the book \cite{Brown-Ozawa}, or \cite{Junge-Ruan,MR2106828}.

The purpose of the present work is to take a step in the opposite direction and study approximation properties which are weaker than the OAP, and are thus enjoyed by more operator spaces.
The original definition of the OAP 
\cite{Effros-Ruan-Approximation} had many parallels with various approaches to the classical AP, most notably the ones related to tensor products, but was lacking a corresponding notion of ``noncommutative compact set''.
The answer was provided by Webster \cite{Webster,webster1998matrix}, who found the appropriate notion (see Section \ref{Section preliminares} for more details).

Modifying Webster's noncommutative notion of compatness in a manner analogous to the $p$-compact sets of Sinha and Karn, the concepts of $p$-compactness and $p$-compact mappings in operator spaces were introduced and explored by the authors in \cite{ChaDiGa-Operator-p-compact, chavez2024revisiting}.
These concepts naturally lead to the introduction of a $p$-approximation property in the non-commutative setting, which is the main subject of the present work and which will be referred to as the $p$-operator approximation property ($p$-OAP).  The $p$-OAP is a new notion that complements the existing literature on operator space approximation properties.
This property, like the classical approximation property, involves approximating the identity by a net of finite-rank operators in a topology defined by seminorms linked to operator $p$-compact matrix sets.
It is worth mentioning that there was already another notion of non-commutative approximation property introduced by Yew, which is relative to a specific operator space. We relate our new definition to this earlier one.
Our definition of $p$-OAP is a generalization of the classical $p$-AP, in the sense that if an operator space has the former then it must have the latter as a Banach space (Proposition \ref{prop-p-AOP-implies-p-AP}).

If a matrix set is operator $q$-compact, then it is also operator $p$-compact for $p>q$ \cite{chavez2024revisiting}. This implies that as the parameter $p$ decreases, there are fewer operator $p$-compact matrix sets. Consequently, achieving the $p$-OAP for a lower $p$ becomes easier.
 Similar to what happens in the classical setting, we show that every operator space has the $2$-OAP (Corollary \ref{cor-2-OAP}).

In this article, we present various equivalences for a space to have the $p$-OAP, relating different topologies (Propositions \ref{prop-uniform-convergence-over-p-compact} and \ref{prop-p-OAP-density-in-CB}, Theorems \ref{thm-characterizations-p-OAP}, \ref{thm-p-OAP-and-quasi-p-nuclear}, and \ref{thm-p-OAP-and-weakly-p-compact}).
In particular, we show that $V$ has the $p$-OAP if and only if any operator $p$-compact mapping into $V$ can be approximated, in the completely bounded norm, by finite-rank mappings. 
 Recall that in the realm of Banach spaces, compactness for a set is characterized by multiple equivalent definitions. In our context, each of these definitions has its counterpart, yet their interrelations are not straightforward, adding complexity to the field.
In the context of multipliers on the reduced $C^*$-algebra of a discrete group, it is a well-known general fact that if a multiplier can be approximated by a finite-rank map then it can in fact be approximated by a finitely supported multiplier (see Lemma \ref{lemma-from-finite-rank-to-finitely-supported} below for a precise statement).
Using these ideas we prove a variation of \cite[Cor. 3.9]{He-Todorov-Turowska}, showing that the $p$-OAP for the reduced $C^*$-algebra of a discrete group implies that operator $p$-compact Herz-Schur multipliers can be approximated in $\cb$-norm by finitely supported multipliers (Corollary \ref{cor-multipliers}).

Recall that Tomiyama's notion of slice mapping properties \cite{Tomiyama}, which has significant applications in operator algebra theory, has been used by Effros and Ruan to characterize both the OAP and its strong version \cite[Cor. 11.3.2]{Effros-Ruan-book}. We prove an analogous result for the $p$-OAP by introducing a variant of the slice mapping property (Theorem \ref{thm-slice-mapping}).
We then relate this latter notion to a variant of the relative Fubini product (Proposition \ref{prop-Fubini-as-kernel}), which allows us to understand the $p$-OAP in terms of a condition involving short exact sequences which is reminiscent of exactness (Corollary \ref{cor-p-OAP-with-exact-sequences}). This approach yields an alternative proof, conceptually clearer, of the fact that every operator space has $2$-OAP (Theorem \ref{th-homogeneous-hilbertian}).

Another well-known property in the classical setting is  that a dual space $X'$ having the approximation property (AP) forces the space $X$ to have it, and the same holds true for the $p$-AP \cite{Choi-Kim}.
In the non-commutative context, a similar result is valid for the OAP \cite[Cor. 11.2.6]{Effros-Ruan-book}. Hence, the question arises about what happens for the $p$-OAP. We demonstrate a transference result, albeit under an additional assumption (Theorem \ref{thm-transfer-p-OAP-from-dual}).
Moreover, one would expect at least that the $p$-OAP is inherited from the bidual $V''$ to $V$. For this, one needs to relate finite-dimensional subspaces of the bidual to those of the ambient space, which can be done  by requiring a hypothesis of local reflexivity (Corollary \ref{cor: tranference from V''}).

\section{Notation and preliminaries} \label{Section preliminares}

We only assume familiarity with the basic theory of operator spaces; the books \cite{Pisier-OS-theory} and \cite{Effros-Ruan-book} are excellent references. Our notation follows closely that from \cite{Pisier-Asterisque, Pisier-OS-theory}, with one notable exception: we denote the dual of a space $V$ by $V'$.

Throughout the article, $V$ and $W$ denote operator spaces. For each $n$, $M_n(V)$ represents the space of $n\times n$ matrices with elements from $V$. We denote the $n$-amplification of a linear mapping $T:V\to W$ as $T_n:M_n(V)\to M_n(W)$. The space of completely bounded linear mappings from $V$ to $W$ is denoted by $\CB(V,W)$, with the subspace of finite-rank mappings represented by $\mathcal{F}(V,W)$. We write $M_\infty(V)$ to indicate the space of all infinite matrices with coefficients in $V$ such that the truncated matrices are uniformly bounded. Our notation for the minimal and projective operator space tensor products is respectively $\otimes_{\min}$ and $\otimes_{\proj}$.
The symbol $\widehat\otimes$ indicates the completion of the tensor product under consideration.
The canonical completely isometric embedding into the bidual is denoted by $\iota_V : V \to V''$.

A linear map $Q : V \to W$ between operator spaces is called a \emph{complete 1-quotient}  if it is onto and the associated map from $V/\ker(Q)$ to $W$ is a completely isometric isomorphism.
In \cite[Sec. 2.4]{Pisier-OS-theory}, these maps are referred to as complete metric surjections. It is proven therein that a linear map $T:V \to W$ is a complete 1-quotient if and only if its adjoint $T' : W' \to V'$ is a completely isometric embedding.

The $\ell_p$ spaces are essential in defining and studying $p$-compactness in Banach spaces. The noncommutative analog of $\ell_p$ is the Schatten class $\mathcal{S}_p$. For $1 \le p < \infty$, $\mathcal{S}_p$ comprises all compact mappings $T$ on $\ell_2$ such that $\tr |T|^p < \infty$, equipped with the norm $\n{T}_{\mathcal{S}_p} = \big( \tr |T|^p \big)^{1/p}$. For $p=\infty$, $\mathcal{S}_\infty$ denotes the space of all compact mappings on $\ell_2$ with the operator norm.
The analogy  is shown by identifying $\ell_p$  with the diagonal mappings in $\mathcal{S}_p$,  noting that any two diagonal mappings commute.
The operator space structure on $\mathcal{S}_p$ is provided through Pisier's theory of complex interpolation for operator spaces \cite[Sec. 2]{Pisier-interpolation},  \cite[Sec. 2.7]{Pisier-OS-theory}. Noting that
 $\mathcal{S}_\infty$ has a canonical operator space structure as it is a $C^*$-algebra  \cite[p. 21]{Effros-Ruan-book} and $\mathcal{S}_1'= \mathcal{B}(\ell_2)$
 also possesses a natural operator space structure  via duality, it is possible to give an operator space structure to each intermediate $\mathcal{S}_p$ space. 
As remarked in \cite[p. 141]{Pisier-OS-theory}, this abstract approach realizes $\mathcal{S}_p$ as a subspace of a $\mathcal B(H)$ space in a highly nonstandard way.
More generally, an operator space $V$ yields a $V$-valued version of $\mathcal{S}_p$, denoted by $\mathcal{S}_p[V]$:
$\mathcal{S}_\infty[V]$ is the minimal operator space tensor product of $\mathcal{S}_\infty$ and $V$,
$\mathcal{S}_1[V]$ is the operator space projective tensor product of $\mathcal{S}_1$ and $V$,
and once again in the case $1 < p < \infty$ we define $\mathcal{S}_p[V]$ via complex interpolation between $\mathcal{S}_\infty[V]$ and $\mathcal{S}_1[V]$  \cite{Pisier-Asterisque}.
For $1 < p \le \infty$, the dual of $\mathcal{S}_p[V]$ can be canonically identified with $\mathcal{S}_{p'}[V']$, where $p'$ satisfies $1/p +  1/p' = 1$ \cite[Cor. 1.8]{Pisier-Asterisque}.
In the discussion above, if we replace $\mathcal{S}_1$ by the space $\mathcal{S}_1^n$ of $n\times n$ matrices with the trace norm, and $\mathcal{S}_\infty$ by the space $M_n$, we can analogously construct operator spaces $\mathcal{S}_p^n$ and $\mathcal{S}_p^n[V]$.
We will often consider the elements of the spaces $\mathcal{S}_p \widehat\otimes_{\min} V$ and $\mathcal{S}_p[V]$ as infinite matrices with entries in $V$.
In the first case the meaning is clear: since $\mathcal{S}_p \widehat\otimes_{\min} V$ completely isometrically embeds into $\CB(\mathcal{S}_p',V)$, we identify $v \in \mathcal{S}_p \widehat\otimes_{\min} V$ with the infinite $V$-valued matrix that arises from applying $v$ (considered as a map $\mathcal{S}_p' \to V$) to the matrix units in $\mathcal{S}_p'$.
For $v \in \mathcal{S}_p[V]$ the identification as an infinite matrix is not immediately clear, since $\mathcal{S}_p[V]$ was constructed using complex interpolation.
The reader is invited to check \cite[pp. 18--20]{Pisier-Asterisque} for further details.

A \emph{mapping ideal} $(\mathfrak{A},\mathbf{A})$ is an assignment, for each pair of operator spaces $V,W$, of a linear space $\mathfrak{A}(V,W) \subseteq \CB(V,W)$  together with an operator space structure $\mathbf{A}$ on $\mathfrak{A}(V,W)$ such that
\begin{enumerate}[(a)]
\item The identity map $\mathfrak{A}(V,W) \to \CB(V,W)$ is a complete contraction.
\item For every $v'\in M_n(v')$ and $w\in M_m(w)$ the mapping $v'\otimes w$ belongs to $M_{nm}(\mathfrak{A}(V,W))$ and $\mathbf{A}_{nm}(v'\otimes w)=\|v'\|_{M_n(V')} \|y\|_{M_m(W)}$.
\item The ideal property: whenever $T \in M_n(\mathfrak{A}(V,W))$, $R \in \CB(V_0,V)$ and $S \in \CB(W,W_0)$, it follows that $S_n \circ T \circ R$ belongs to $ M_n(\mathfrak{A}(V_0,W_0))$ with
$$
\mathbf{A}_n( S_n \circ T \circ R ) \le \n{S}_{\cb} \mathbf{A}_n(T) \n{R}_{\cb}.
$$
\end{enumerate}
Note that this is the definition of \cite[Def. 7.1.1]{ChaDiGa-tensor-norms}, which is stronger than that of \cite[Sec. 12.2]{Effros-Ruan-book} (because of the item (b)).
All of the mapping ideals considered in the present paper have been checked to satisfy this stronger definition, see \cite[Chap. 7]{ChaDiGa-tensor-norms}.

We now review some fundamental definitions focusing on operator approximation properties. An operator space \(V\) has the \emph{Operator Approximation Property} (OAP) if there exists a net of finite-rank mappings $T_{\alpha} \in \CB(V,V)$ converging to the identity mapping $Id_V$ in the stable point-norm topology (i.e., the weakest topology in which the seminorms $$T \mapsto \Vert Id_{ \mathcal S_{\infty}}\otimes T(v) \Vert, \;\; v \in \mathcal S_{\infty} \widehat{\otimes}_{\min} V$$ are continuous).
Equivalently, for each $v_1, \dots, v_n \in \mathcal S_{\infty} \widehat{\otimes}_{\min} V$ and $\varepsilon >0$, there is a finite-rank mapping $T\in \CB(V,V)$ such that
$$
\Vert Id_{ \mathcal S_{\infty}} \otimes T(v_k) - v_k \Vert < \varepsilon,
$$
for all $k=1, \dots, n$.

An operator space has the \emph{Strong Operator Approximation Property} (SOAP) if for any $v_1, \dots, v_n \in \mathcal{B}(\ell_2) \widehat{\otimes}_{\min} V$ and $\varepsilon >0$, there is a finite-rank mapping $T\in \CB(V,V)$ such that
$$
\Vert Id_{\mathcal B(\ell_2)} \otimes T(v_k) - v_k \Vert < \varepsilon,
$$
for all $k=1, \dots, n$.

An operator space \(V\) has the \emph{Completely Bounded Approximation Property} (CBAP) if there exists a constant \(C \geq 1\) and a net of finite-rank mappings \(T_{\alpha} \in \CB(V,V)\) such that \(\|T_{\alpha}\|_{\cb} \leq C\) and \(T_{\alpha} \to \text{id}_V\) pointwise.

A \emph{matrix set $\mathbf{K} = (K_n)_n$ over an operator space $V$} is a sequence of subsets $K_n \subseteq M_n(V)$ for each $n\in\N$. A typical example of a matrix set over $V$ is the \emph{closed matrix unit ball of $V$} given by
$\mathbf{B}_V = \big( B_{M_n(V)} \big)_n$. For a linear map $T : V \to W$ between operator spaces, the expression $T(\mathbf{K})$ denotes the matrix set $\big(T_n(K_n)\big)_n$ where $T_n$ is the $n$-th amplification of $T$. For two matrix sets $\mathbf{K} = (K_n)_n$ and $\mathbf{L} = (L_n)_n$ defined over the same operator space $V$, we denote $\mathbf{K} \subseteq \mathbf{L}$ to signify that $K_n \subseteq L_n$ holds for all $n \in \mathbb{N}$.

Note that the language of matrix sets allows us to more transparently see the analogy between bounded and completely bounded linear maps. A linear map $T : X \to Y$ between Banach spaces is bounded with norm at most $C$ if and only if $T(B_X) \subseteq C B_Y$, whereas a linear map $T : V \to W$ between operator spaces is completely bounded with completely bounded norm at most $C$ if and only if $T(\mathbf{B}_V) \subseteq C \mathbf{B}_W$.

\subsection{Operator $p$-compactness}

In this subsection, we introduce the notion of operator $p$-compact mappings. We first recall \cite[Lem. 3.2]{chavez2024revisiting} which says that each element of $\mathcal{S}_p[V]$ can be seen as a completely bounded mapping from $\mathcal{S}_p'$ to $V$.

\begin{lemma}\label{lemma-Theta-well-defined}
The formal identity $\mathcal{S}_p[V] \to \mathcal{S}_p\widehat\otimes_{\min} V \hookrightarrow \CB(\mathcal{S}_p',V)$ is a contractive injection. Moreover, the image of $v \in \mathcal{S}_p[V]$ is the map $\Theta^v : \mathcal{S}_p' \to V$ given by $(\alpha_{ij}) \mapsto \lim_{N\to\infty}\sum_{i,j=1}^{N} \alpha_{ij} v_{ij}$.
\end{lemma}

This allow us to present the following definition for matrix sets.

A matrix set $\mathbf{K}=(K_n)$ over $V$ is \emph{relatively operator $p$-compact} if there exists $v \in \mathcal{S}_p[V]$ such that $\mathbf{K}\subset \Theta^v(\mathbf{B}_{\mathcal{S}_p'})$.  

In the case $p=\infty$, we just say \emph{operator compact matrix set} instead of \emph{operator $\infty$-compact matrix set.}

In the context of Banach spaces, a $p$-compact map is defined as a mapping that sends the unit ball into a relatively $p$-compact set. In \cite{ChaDiGa-Operator-p-compact, chavez2024revisiting} we study the extension of this concept to the non-commutative framework.  The analogy is given by mapping the matrix unit ball into a relatively operator $p$-compact matrix set.

\begin{definition} \label{defn: operator p-compact maps}
A completely bounded mapping  $T: V \to W$  is \emph{operator $p$-compact} ($1\le p\le \infty$) if $T(\mathbf{B}_V)$ is a relatively operator $p$-compact matrix set in $W$. The operator $p$-compact norm of $T$ is defined as   
\begin{equation}\label{kpo}
\kappa_p^o(T) = \inf\{ \Vert w \Vert_{\mathcal{S}_p[W]} : T(\mathbf{B}_V) \subset \Theta^w(\mathbf{B}_{\mathcal{S}_p'})\}.
\end{equation}
\end{definition}

In the field of Banach spaces, the notion of a set being compact has multiple equivalent and useful formulations. However, when this concept is generalized to the non-commutative context, distinct and non-equivalent definitions emerge. 
One such generalization is the idea of operator compact matrix sets, which we have already referred to. 
Another one is the concept of completely compact matrix set defined by Webster in \cite[Def. 4.1.2]{Webster}:

\begin{definition}
    A matrix set $\mathbf{K}=(K_n)$ over $V$ is \emph{ completely compact} if it is closed, completely bounded  and for every $\varepsilon>0$ there exists a finite-dimensional $V_\varepsilon \subseteq V$ such that for each $n\in\N$ and $x\in K_n$ there is  $v\in M_n(V_\varepsilon) $ with $\|x-v\|\le \varepsilon$.
    A linear mapping $T : V \to W$ between operator spaces is called \emph{completely compact} if $\overline{T(\mathbf{B}_V)}$ is completely compact.
\end{definition}

 It is important to highlight \cite{Webster} that this definition is weaker; specifically, if a mapping is operator compact, then it is also completely compact.
 
The following  result appeared in \cite[Thm. 3.15]{chavez2024revisiting}. We state it here because it will be useful to our purposes.

\begin{theorem}\label{thm-factorization-Choi-Kim}
Any operator $p$-compact mapping $T: V \to W$ can be factored as $AT_0B$ where $A$ is operator compact, $T_0$ is operator $p$-compact, and $B$ is completely compact. Moreover, $\kappa_p^o(T) = \inf\{ \n{A}_{\cb} \kappa_p^o(T_0) \n{B}_{\cb} \}$ where the infimum is taken over all such factorizations.
In the case $p=\infty$, we can moreover take $B$ to be operator compact. 
\end{theorem}

A weaker notion of $p$-compactness which has its roots in the classical realm, was also introduced in \cite[Def. 5.1]{chavez2024revisiting}. To state the definition we need first to recall a  non-commutative version of the sequence space $\ell_p^w(X)$:

$$ S_{p}^w[V]:=\{ v \in M_{\infty}(V) :   \sup_{N} \Vert (v_{ij})_{i,j=1}^N  \Vert_{S_p^N \widehat{\otimes}_{\min} V } < \infty \}. $$

This is an operator space endowed with the matricial norm structure given by

$$ \Vert  \big ((v_{ij}^{kl})_{i,j} \big)_{k,l=1}^n  \Vert_{M_n(S_{p}^w[V])} := \sup_{N} \Vert \big((v_{ij}^{kl})_{i,j=1}^N \big)_{k,l=1}^{n} \Vert_{M_n(S_p^N \widehat{\otimes}_{\min} V) }.$$

For $1 \leq p \leq \infty$, we can equivalently define $S_{p}^w[V]$  through the following completely isometric identification \cite[Lem. 2.4]{ChaDiGa-Operator-p-compact} $$S_{p}^w[V] = \CB(S_{p'},V).$$

\begin{definition}
Let $1 \le p \le \infty$.
 A matrix set $\mathbf{K}=(K_n)$ over $V$ is called \emph{relatively operator weakly $p$-compact} if there exists $v \in \mathcal{S}_p^w[V]$  
 such that $\mathbf{K}\subset \Theta^v(\mathbf{B}_{\mathcal{S}_p'})$. 
A linear map $T : V \to W$ is called \emph{operator weakly $p$-compact} if the matrix set $T(\mathbf{B}_V)$ is relatively operator weakly $p$-compact, and in this case we define the operator weakly $p$-compact norm of $T$ by
\[
\omega_p^o(T) := \inf\left\{\| v\|_{ \mathcal{S}_p^w[V] } \;:\; v\in  \mathcal{S}_p^w[V], \; T(\mathbf{B}_V)\subset \Theta^v(\mathbf{B}_{\mathcal{S}_p'}) \right\}.
\]
We denote by $\mathcal{W}_p^o(V,W)$ the set of all operator weakly $p$-compact maps from $V$ to $W$.
\end{definition}

\section{The $p$-Operator Approximation Property} \label{sec p-oap}

In Banach spaces, the $p$-approximation property has been studied by many researchers in the field, reflecting a substantial interest in the area (see, for instance, the pioneering works  \cite{Choi-Kim, Delgado-Oja-Pineiro-Serrano,lassalle2012p}). In this section, we delineate its non-commutative counterpart. To facilitate this, we introduce a specific topology that is inspired both in the definitions given in the classical context and the one employed in defining the operator approximation property.
For that note, on the one hand, that   $T \in \CB(V,W)$ induces a completely bounded mapping $Id_{\mathcal{S}_p} \otimes T: \mathcal{S}_p \widehat{\otimes}_{\min} V \to \mathcal{S}_p \widehat{\otimes}_{\min} W$. On the other hand, by \cite[Cor. 1.2]{Pisier-Asterisque}, we can reinterpret this assignment by restricting its domain and codomain to get 
$Id_{\mathcal{S}_p} \otimes T: \mathcal{S}_p [V] \to \mathcal{S}_p [W]$. In turn, due to Lemma \ref{lemma-Theta-well-defined}, this yields a continuous map (also referred to by the same name)  $Id_{\mathcal{S}_p} \otimes T: \mathcal{S}_p[V] \to \mathcal{S}_p \widehat{\otimes}_{\min} W$.

\begin{definition}
Let $V$ and $W$ be operator spaces.
Let $1\le p\le \infty$. We define the topology $\tau_p$ on $\CB(V,W)$ to be the weakest topology in which the seminorms
\[
T \mapsto \|(Id_{\mathcal{S}_p} \otimes T)(v)\|_{\mathcal{S}_p \widehat{\otimes}_{\min} W}, \qquad v \in \mathcal{S}_p[V]
\]
are continuous.

We say that $V$ has the \emph{$p$-Operator Approximation Property} ($p$-OAP, for short) if there exists a net $(T^\alpha)$ in $\mathcal{F}(V,V)$ converging to the identity map $Id_V$ in this topology.
\end{definition}

By standard arguments, $V$ has the $p$-OAP if and only if for every $\varepsilon>0$ and every $v_1,v_2,\dotsc,v_n \in \mathcal{S}_p[V]$ there is  $T \in \mathcal{F}(V,V)$ such that for $k=1,2,\dotsc,n$ we have
\begin{equation}\label{eqn-approx-of-identity}
    \|(Id_{\mathcal{S}_p} \otimes T)(v_k) - v_k \|_{\mathcal{S}_p \widehat{\otimes}_{\min} V} < \varepsilon.
\end{equation}

The sum $v_1 \oplus \cdots \oplus v_n$ can be clearly identified with a single element $v \in \mathcal{S}_p[V]$, so $(Id_{\mathcal{S}_p} \otimes T)(v)$ is then understood as $(Id_{\mathcal{S}_p} \otimes T)(v_1) \oplus \cdots \oplus (Id_{\mathcal{S}_p} \otimes T)(v_n)$. In this case, since the restrictions to blocks are complete contractions $\mathcal{S}_p \to \mathcal{S}_p$, and when tensored with $Id_V$ we get a complete contraction $\mathcal{S}_p \widehat{\otimes}_{\min} V\to \mathcal{S}_p \widehat{\otimes}_{\min} V$,  for each  $k=1,2,\dotsc,n$ we have
\[
\|(Id_{\mathcal{S}_p} \otimes T)(v_k) - v_k \|_{\mathcal{S}_p \widehat{\otimes}_{\min} V} \le \|(Id_{\mathcal{S}_p} \otimes T)(v) - v \|_{\mathcal{S}_p \widehat{\otimes}_{\min} V}.
\]
Therefore, for $V$ to have the $p$-OAP it suffices to have the approximation \eqref{eqn-approx-of-identity} with $n=1$.

In the case of $p=\infty$, this corresponds to the classical Operator Approximation Property (OAP).

Yew \cite{Yew} defined an approximation property related to a fixed operator space $Z$. Namely, $V$ has the \textit{$Z$-approximation property ($Z$-AP)} if for every $\varepsilon>0$ and every $u\in Z\widehat{\otimes}_{\min} V$ there is $T\in\mathcal F(V,V)$ such that $\|(Id_{Z} \otimes T)(u) - u \|_{Z \widehat{\otimes}_{\min} V} < \varepsilon.$
Not that our $p$-OAP is weaker than Yew's $\mathcal{S}_p$-AP, since  we consider only $v \in \mathcal{S}_p[V]$ rather than $v \in \mathcal{S}_p \widehat{\otimes}_{\min} V$ (and therefore, for $p=\infty$, both notions coincide because $\mathcal{S}_\infty[V]=\mathcal{S}_\infty\widehat\otimes_{\min} V$).

We remark that the terminology $p$-OAP had already been used in \cite{An-Lee-Ruan} for a different approximation property. However, there is no risk of confusion as their notion is for $p$-operator spaces and not traditional operator spaces.

We now observe that the $p$-Operator Approximation Property implies the classical $p$-Approximation Property ($p$-AP), as expected, which further justifies the chosen nomenclature.

\begin{proposition}\label{prop-p-AOP-implies-p-AP}
Let $1 \le p \le \infty$. If an operator space $V$ has $p$-OAP, then it has $p$-AP as a Banach space.    
\end{proposition}

\begin{proof}
Let $x = (x_n)_{n=1}^\infty \in \ell_p(V)$ (or $c_0(V)$ in the case $p=\infty$). 
By identifying it with a diagonal matrix, we can think of $x$ as an element of $\mathcal{S}_p[V]$ \cite[Cor. 1.3]{Pisier-Asterisque}.
Therefore, given $\varepsilon>0$ there exists $T \in \mathcal{F}(V,V)$ such that $\n{ (Id_{\mathcal{S}_p} \otimes T)(x)-x }_{\mathcal{S}_p \widehat{\otimes}_{\min} V} < \varepsilon$.
Note that $y = (Id_{\mathcal{S}_p} \otimes T)(x)-x$ is still a diagonal matrix.
Moreover, since $\mathcal{S}_p \widehat{\otimes}_{\min} V\subset\CB(\mathcal{S}_{p'},V)$, $y$ defines a bounded map from $\ell_{p'}$ to $V$. Therefore,
\[
\n{ (T (x_n)) - (x_n) }_{\ell_{p'}^w[V]} = \n{y}_{\mathcal{B}(\ell_{p'},V)} \le \n{y}_{\CB(\mathcal{S}_{p'},V)} < \varepsilon
\]
which implies that $V$ has the $p$-AP \cite[Comment between Defn. 6.1 and Prop. 6.2]{Sinha-Karn}.
\end{proof}

For any $2<p\le \infty$ it is known that there exists a Banach space without $p$-AP \cite[Thm. 6.7]{Sinha-Karn}, and therefore by Proposition \ref{prop-p-AOP-implies-p-AP} there exists an operator space without $p$-OAP: simply endow the aforementioned Banach space example with any compatible operator space structure. 
It would be interesting to find an operator space that has the $p$-AP but fails the $p$-OAP, although such a construction would likely be delicate.
In the case $p=\infty$, it is known that such a thing can happen in a rather extreme way: Arias \cite{Arias} has shown that there exist Hilbertian operator spaces without the OAP, see also \cite{OikhbergRicard2004MathAnn} for related examples.

As in the classical setting, the topology defining the $p$-OAP can be understood in terms of uniform convergence on (operator) $p$-compact (matrix) sets.

\begin{proposition}\label{prop-uniform-convergence-over-p-compact}
Let $V$ and $W$ be operator spaces.
Let $1\le p\le \infty$. The topology $\tau_p$ on $\CB(V,W)$ coincides with the weakest topology induced by the seminorms
\[
T \mapsto \|T\|_{\mathbf{K}} := \sup \{ \| T_n(x) \|_{M_n(W)} \;:\; x \in K_n,\, n\in\N \}
\]
where $\mathbf{K} = (K_n)_{n=1}^\infty$ is an operator $p$-compact matrix set over $V$.
In other words, the topology $\tau_p$ is the topology of uniform convergence on  relatively operator $p$-compact matrix sets.
\end{proposition}

\begin{proof}
From the definition of relatively operator $p$-compact matrix set, note that the topology in the statement does not change if we restrict to the case $\mathbf{K} = \Theta^v(\mathbf B_{\mathcal{S}_p'})$ for some $v \in \mathcal{S}_p[V]$.

Consider $n\in\N$ and $\sigma \in M_n(\mathcal{S}_{p'})$ with $ \Vert \sigma \Vert_{M_n(\mathcal{S}_{p'})} \leq 1$.
By \cite[Cor.1.2]{Pisier-Asterisque} there is a natural way to associate to $T$ a completely bounded mapping $Id_{\mathcal{S}_p} \otimes T:\mathcal{S}_p[V]\to \mathcal{S}_p[W]$ with $\|T\|=\|Id_{\mathcal{S}_p} \otimes T\|$ 
We have
$T_n(\Theta_n^v(\sigma))=\Theta^{w}_n(\sigma)$ with $w=Id_{\mathcal{S}_p} \otimes T(v)$
and the quantity $\|T\|_{\mathbf{K}}$ is precisely the supremum of the norms in $M_n(W)$ of these expressions. In other words, $\n{T}_{\mathbf{K}} = \|\Theta^{w}\|_{\CB(\mathcal{S}_p',W)}$.

On the other hand, the quantity $\|(Id_{\mathcal{S}_p} \otimes T)(v)\|_{\mathcal{S}_p \widehat{\otimes}_{\min} W}$ is the $\cb$-norm of the associated mapping $\mathcal{S}_{p'} \to W$, 
which is, precisely, $\Theta^{w}$,
from where the desired conclusion follows.
\end{proof}

In fact, note that the previous proof gives the following more precise statement: 
if $\mathbf{K} = \Theta^v(\mathbf B_{\mathcal{S}_p'})$ for some $v \in \mathcal{S}_p[V]$, then for any $T \in \CB(V,W)$ we have 
\begin{equation}\label{eqn-equality-norm-K-norm-Sp-min}
\n{T}_{\mathbf{K}} = \|(Id_{\mathcal{S}_p} \otimes T)(v)\|_{\mathcal{S}_p \widehat{\otimes}_{\min} W}.     
\end{equation}

Since every relatively operator $q$-compact matrix set is operator $p$-compact (for $q<p$) \cite[Prop. 3.12]{chavez2024revisiting}, as a direct consequence of Proposition  \ref{prop-uniform-convergence-over-p-compact} we get the next monotonicity result:

\begin{proposition}\label{prop-monotonicity-p-OAP}
 Let $1\le q<p\le \infty$.  If an operator space $V$ has the $p$-OAP then it also has the $q$-OAP.   
\end{proposition} 

The following is a generalization of the analogous result for the OAP \cite[Lemma 11.2.1]{Effros-Ruan-book}, and the proof is the same just using the topology $\tau_p$ instead of $\tau_\infty$.

\begin{proposition}\label{prop-p-OAP-density-in-CB}
    Let $V$ be an operator space
and $1\le p\le \infty$.
The following are equivalent:
\begin{enumerate}[(i)]
    \item $V$ has $p$-OAP.
    \item $\mathcal{F}(V,V)$ is $\tau_p$-dense in $\CB(V,V)$.
    \item For every operator space $W$, $\mathcal{F}(V,W)$ is $\tau_p$-dense in $\CB(V,W)$.
    \item For every operator space $W$, $\mathcal{F}(W,V)$ is $\tau_p$-dense in $\CB(W,V)$.
\end{enumerate}
\end{proposition}

In Banach spaces, when dealing with bounded nets of mappings, different notions of convergence coincide. As expected, many similar equivalences also occur in our non-commutative framework. We present now some such results which will be useful for the characterization of the $p$-OAP.

First, we state a generalization of \cite[Cor. 6.2]{webster1998matrix}.

\begin{lemma}\label{lemma-finite-matrix-ball}
For any $1\le p \le \infty$, the matrix unit ball of a finite-dimensional operator space is operator $p$-compact.
\end{lemma}

\begin{proof}
We follow the same strategy as in the proof of \cite[Cor. 6.2]{webster1998matrix}.
First, note that for any $n\in\N$, the matrix unit ball of $\mathcal{S}_{p'}^n$ is operator $p$-compact:
the identity map $\mathcal{S}_{p'}^n \to \mathcal{S}_{p'}^n$ is the map $\Theta^v$ corresponding to some $v \in \mathcal{S}_p^n[\mathcal{S}_{p'}^n]$.
Since for each $n\in\N$ the spaces $\mathcal{S}_{p'}^n$ and $\mathcal{S}_1^n$ are completely isomorphic, the matrix unit ball of the latter is operator $p$-compact.
Finally, since any finite-dimensional operator space is completely isomorphic to a quotient of an $\mathcal{S}_1^n$ space, the desired conclusion follows. 
\end{proof}

We denote by $\tcc$ the topology of uniform convergence on completely compact matrix sets.
As previously mentioned, for bounded nets of mappings, we now see that convergence in several different topologies actually coincides. 

\begin{proposition}\label{proposition-bounded-net}
Let $V$ and $W$ be  operator spaces, $T_\alpha, T\in \CB(V,W)$, and $1 \le p < \infty$. If the net $(T_\alpha)$ is  bounded, the following are equivalent:
\begin{enumerate}[(i)]
    \item $(T_\alpha)$ point-norm-converges to $T$.
    \item $(T_\alpha)$ $\tau_\infty$-converges to $T$. 
     \item $(T_\alpha)$ $\tau_p$-converges to $T$.
     \item $(T_\alpha)$ $\tau_{cc}$-converges to $T$. 
\end{enumerate}
\end{proposition}

\begin{proof}
    $(i)\Rightarrow (ii)$ By \cite[Remark 6.1.1]{ChaDiGa-tensor-norms} if the bounded net $(T_\alpha)$ converges to $T$ in the  point-norm topology then it converges in the stable point-norm topology. That is, for every $v \in \mathcal{S}_\infty[V]$ we have that $\n{(Id_{\mathcal{S}_\infty}\otimes (T_\alpha-T))(v)}_{\mathcal{S}_\infty[W]}$ converges to 0. Note that, since $\mathcal{S}_\infty[W]=\mathcal{S}_\infty\widehat\otimes_{\min} W$, this quantity is nothing but $\n{\Theta^{w_\alpha}}_{\CB(\mathcal{S}_1,W)}$ with $w_\alpha=Id_{S_\infty}\otimes (T_\alpha - T)(v)\in \mathcal{S}_\infty[W]$. Now, to verify that $(T_\alpha)$ converges to $T$
 uniformly on relatively operator compact matrix sets, it suffices to take $\mathbf{K} = \Theta^v(\mathbf{B}_{\mathcal{S}_1})$ for some $v \in \mathcal{S}_\infty[V]$.
For each $x \in \mathbf{K}$  there exists $y \in B_{M_n(\mathcal{S}_1)}$ such that $x = (\Theta^v)_n y$ for certain $n$. Therefore,
\[
\n{(T_\alpha)_nx-T_nx} _{M_n} = \n{\big(\Theta^{w_\alpha}\big)_ny}_{M_n} \le \n{\Theta^{w_\alpha}}_{\CB(\mathcal{S}_1,\C)}
\]
which implies that $(T_\alpha)$ converges to $T$
 uniformly on $\mathbf{K}$.

  $(ii)\Rightarrow (iii)$ It is clear, since any operator $p$-compact matrix set is operator compact.

    $(iii)\Rightarrow (iv)$ Suppose that $\mathbf{L}$ is a completely compact matrix set over $V$. Without loss of generality, assume $\mathbf{L} \subseteq \mathbf{B}_V$ and $\|T_\alpha\|_{\cb}\le 1$, $\n{T}_{\cb}\le 1$.
Given $\varepsilon>0$, there exists a finite-dimensional $E \subseteq V$ such that every point in $\mathbf{L}$
 is $\varepsilon/2$-close to a point in $\mathbf{B}_E$. By Lemma \ref{lemma-finite-matrix-ball},  $\mathbf{B}_E$ is  operator $p$-compact so, by hypothesis, for $\alpha$ large enough $T_\alpha$ and $T$ are $\epsilon/2$-close on $\mathbf{B}_E$, which implies that $T_\alpha$ and $T$ are $\varepsilon$-close on $\mathbf{L}$.

 $(iv) \Rightarrow (i)$ This implication follows from the fact that for each $v\in V$, the matrix set $\mathbf K$, where $K_1=\{v\}$ and $K_n=\emptyset$ (for every $n\ge 2$) is completely compact.
\end{proof}

Now we present a noncommutative version of (a part of)  \cite[Thm. 2.1]{Delgado-Oja-Pineiro-Serrano} and \cite[Thm. 4.2]{Choi-Kim}. 
 Surprisingly, in part $(iv)$, we need to employ the topology $\tau_{cc}$ instead of the perhaps expected $\tau_{\infty}$ (uniform convergence on operator compact matrix sets). For the case $p=\infty$, it should be noted that Webster \cite[Thm. 4.4]{webster1998matrix} already proved the equivalence $(i) \Leftrightarrow (iii)$.

\begin{theorem}\label{thm-characterizations-p-OAP}
Let $1 \le p \le \infty$ and let $V$ be an operator space. The following are equivalent:
\begin{enumerate}[(i)]
    \item $V$ has $p$-OAP.
    \item For every operator space $W$, $\mathcal{F}(W,V)$ is $\tau_p$-dense in $\mathcal{K}_\infty^o(W,V)$.
    
    \item For every operator space $W$, $\mathcal{F}(W,V)$ is $\|\cdot\|_{\cb}$-dense in $\mathcal{K}_p^o(W,V)$.
    \item For every operator space $W$, $\mathcal{F}(W,V)$ is $\tcc$-dense in $\mathcal{K}_p^o(W,V)$.
\end{enumerate}
\end{theorem}

\begin{proof}

$(i) \Rightarrow (ii)$ is trivial from Proposition \ref{prop-p-OAP-density-in-CB}.

$(ii) \Rightarrow (iii)$: Fix $T \in \mathcal{K}_p^o(W,V)$. By Theorem \ref{thm-factorization-Choi-Kim} there exist an operator space $Z$, an operator $p$-compact map $R : W \to Z$, and an operator compact map $S : Z \to V$ such that $T = SR$.
By our assumption, since $R(\mathbf{B}_W)$ is operator $p$-compact, given $\varepsilon>0$ there exists $A \in \mathcal{F}(Z,V)$ such that for every $n\in\N$ and every $x \in R_n(B_{M_n(W)})$ we have $\n{ (S-A)_nx }_{M_n(V)} < \varepsilon$.
Now,
\[
\n{T-AR}_{\cb} = \n{(S-A)R}_{\cb} = \sup\left\{ \n{ (S-A)_nx }_{M_n(V)} \;:\;n\in\N, x \in R_n(B_{M_n(W)})  \right\} \le \varepsilon    
\]
which gives the desired conclusion.

$(iii) \Rightarrow (iv)$:
This is clear since every completely compact matrix set is contained in a multiple of the matrix unit ball. 

$(iv) \Rightarrow (i)$:
Let $v \in \mathcal{S}_p[V]$, $\mathbf{K} = \Theta^v (\mathbf B_{\mathcal{S}_p'})$ and $\varepsilon>0$. Note that, trivially from the definition, the mapping $\Theta^v : \mathcal{S}_p' \to V$ is operator $p$-compact.
By the proof of Theorem \ref{thm-factorization-Choi-Kim}  we can factor $\Theta^v = BA$ where $A : \mathcal{S}_p'\to \mathcal{S}_p'$ is completely compact and $B : \mathcal{S}_p' \to V$ is operator $p$-compact.
Now let $W = \mathcal{S}_p'/\ker B$, and let $Q : \mathcal{S}_p' \to W$ be the quotient map. Let $\widehat{B} : W \to V$ be the induced map making $\widehat{B}Q = B$, and note that $\widehat{B}$ is also operator $p$-compact.
Consider the matrix set $\mathbf{H} = (H_n)_n = Q A(\mathbf{B}_{\mathcal{S}_p'})$ over $W$, and note that $\mathbf{H}$ is 
relatively completely  compact because $A$ is completely compact. By our assumption, there exists $T \in\mathcal{F}(W,V)$ such that
\[
\sup \left\{ \| T_nh - \widehat{B}_nh \|_{M_n(V)} \;:\;  n\in\N, h \in H_n \right\} < \varepsilon/2.
\]
The mapping $T$, having finite-rank, can be written as $T = \sum_{k=1}^N y_k' \otimes v_k$, where $y_k' \in W'$, $v_k \in V$ and $\sum_{k=1}^N\|v_k\| = 1$.
We claim that the image of $\widehat{B}'$ is $\tcc$-dense in $W'$. If not, by the Hahn-Banach theorem there exists a nonzero continuous functional $\varphi : (W',\tcc) \to \C$ which vanishes on $\widehat{B}'V'$. Note that $\varphi : (W',w^*) \to \C$ is also continuous. Indeed, it is enough to see \cite[Cor. 4.46]{fabian2001functional} that $\varphi$ is $w^*$-continuous when restricted to $B_{W'}$ and this holds by the equivalences of Proposition \ref{proposition-bounded-net}.
Hence, $\varphi$ corresponds to evaluation at some $w_0\in W\setminus\{0\}$. Therefore, for each $v'\in V'$ we have that $0 = \pair{\varphi}{\widehat{B}'v'} = \pair{\widehat{B}'v'}{w_0} = \pair{v'}{\widehat{B}w_0}$, which implies that $\widehat{B}w_0=0$ contradicting the injectivity of $\widehat{B}$.
Thus, for each $k=1,2,\dotsc, N$ there exists $v_k' \in V'$ such that
\[
\sup \left\{ \| \mpair{ h}{  \widehat{B}'v_k'-y_k' } \|_{M_n} \;:\;  n\in\N, h \in H_n \right\} < \varepsilon/2.
\]
Now define $R = \sum_{k=1}^N v_k' \otimes v_k \in \mathcal{F}(V,V)$.
For every $n\in\N$ and $x \in K_n = (\Theta^v)_nB_{M_n(\mathcal{S}_p')}$, there exists $\sigma \in B_{M_n(\mathcal{S}_p')}$ such that $x = (\Theta^v)_n\sigma = (\widehat{B}QA)_n\sigma$.
Therefore $x = \widehat{B}_nh$ for $h = (QA)_n\sigma \in H_n$, and thus
\begin{multline*}
\|R_nx-x\|_{M_n(V)} \le \|(R\widehat{B})_nh - T_nh\|_{M_n(V)} + \|T_nh-\widehat{B}_nh\|_{M_n(V)} < \|(R\widehat{B})_nh - T_nh\|_{M_n(V)} + \varepsilon/2 \\
\le \sum_{k=1}^N \| \mpair{ \widehat{B}_nh}{ v_k'}- \mpair{h}{y_k'} \|_{M_n} \|v_k\| + \varepsilon/2 \le \max_{1\le k \le N} \| \mpair{h}{ \widehat{B}'v_k' - y_k'} \|_{M_n} + \varepsilon/2 < \varepsilon/2+\varepsilon/2 = \varepsilon,
\end{multline*}
which gives the desired conclusion.

\end{proof}

We would like to emphasize that the equivalence $(i) \Leftrightarrow (iii)$ in Theorem \ref{thm-characterizations-p-OAP} is more satisfactory than the corresponding situation for Yew's $\mathcal{S}_p$-AP, where the analogous equivalence is not known \cite[Thms. 6.4 and 6.6]{Yew} .

We stress that in Theorem \ref{thm-characterizations-p-OAP} we were unable to get an equivalence involving the topology $\tau_\infty$. This is likely due to the fact that in the noncommutative setting there are several notions of compactness—such as complete compactness and operator compactness—which all generalize the classical one.
It would be interesting to know whether there is  an operator space satisfying the density condition in $\tau_\infty$ but the space does not have $p$-OAP.

Similarly to the classical context \cite[Thm. 6.4]{Sinha-Karn} (see also  \cite[Cor. 2.5]{Delgado-Oja-Pineiro-Serrano}), the $2$-OAP holds in every operator space.

\begin{corollary}\label{cor-2-OAP}
Every operator space has the 2-OAP (and, hence, the $p$-OAP for every $1\le p\le 2$ by Proposition \ref{prop-monotonicity-p-OAP}).
\end{corollary}

\begin{proof}
Let $V$ and $W$ be operator spaces and
let $T \in \mathcal{K}^o_2(W,V)$.
From the proof of \cite[Thm. 3.11]{ChaDiGa-Operator-p-compact}, $T$ factors through a quotient of  $\mathcal{S}_2$. Now, by \cite[p. 129]{Pisier-OS-theory} or \cite[Rmk. 1.11]{Pisier-Asterisque}, $\mathcal{S}_2$ is completely isometric to $\OH(\N \times \N)$.
Note that a quotient of an $\OH(I)$ space is itself an $\OH(J)$ space, because its dual is a subspace of $\OH(I)'$, which itself is an $\OH$ space.
Therefore, $T$ factors through an $\OH$ space.
An $\OH$ space has the CBAP, and therefore it has the strong OAP \cite[Thm. 11.3.3]{Effros-Ruan-book}.
By \cite[Sec. 4.4]{Webster}, $\mathcal{F}(W, \OH)$ is $\tcc$-dense in $\CB(W,\OH)$, which then implies that condition $(iv)$ in Theorem \ref{thm-characterizations-p-OAP} is satisfied and therefore $V$ has 2-OAP.
\end{proof}

In fact, note that the previous corollary implies \cite[Thm. 6.4]{Sinha-Karn} because Proposition \ref{prop-p-AOP-implies-p-AP} allows us to transfer approximation properties from an operator space to the underlying Banach space.
For the same reason, \cite[Thm. 6.2]{Sinha-Karn} implies that for every $p>2$ there exists an operator space that fails $p$-OAP.

As a consequence of Proposition \ref{proposition-bounded-net} and Theorem \ref{thm-characterizations-p-OAP} we obtain a version of \cite[Cor. 2.6]{Delgado-Oja-Pineiro-Serrano} in the operator space setting. As is often the case when relating a space to its bidual, an assumption of local reflexivity — which does not necessarily hold for all operator spaces — is required. Recall that an  operator space $V$ is said to be \emph{strongly locally reflexive} if given finite-dimensional subspaces $F\subseteq V''$ and $N \subseteq V'$, and $\varepsilon>0$, there exists a complete isomorphism $T : F \to E \subseteq V$ such that
(a) $\n{T}_{\cb},\n{T^{-1}}_{\cb} < 1+\varepsilon$,
(b) $\pair{Tv}{v'} = \pair{v}{v'}$ for all $v \in F$ and $v' \in N$,
(c) $Tv=v$ for all $v\in F \cap V$.

\begin{corollary} \label{cor: tranference from V''}
    Let $V$ be a strongly locally reflexive operator space and $1\le p\le\infty$. If $V''$ has the $p$-OAP then $V$ has the same property.
\end{corollary}

We omit the proof since it is canonically translated from the classical one.

We can now characterize the $p$-OAP for a dual space,
 an operator space version of \cite[Thm. 2.8]{Delgado-Oja-Pineiro-Serrano}. For this, we need the following definition \cite[Def. 4.3]{chavez2024revisiting}: A mapping $T : V \to W$  is \emph{quasi completely $p$-nuclear} if $j \circ T : V \to Y$ is completely $p$-nuclear \cite[Def. 3.1.3.1]{Junge-Habilitationschrift}, where $j : W \to Y$ is a completely isometric embedding of $W$ into an injective operator space $Y$. The class of all quasi completely $p$-nuclear mappings $T:V \to W$ is denoted by $\mathcal{QN}^o_p(V,W)$.
Once again the proof is canonically translated from the classical one, so we omit it.

\begin{theorem}\label{thm-p-OAP-and-quasi-p-nuclear}
Let $1 \le p \le \infty$ and let $V$ be an operator space. The following are equivalent:
\begin{enumerate}[(i)]
    \item $V'$ has $p$-OAP.
    \item For every operator space $W$, $\mathcal{F}(V,W)$ is $\n{\cdot}_{\cb}$-dense in $\mathcal{QN}_p^o(V,W)$.
\end{enumerate}
\end{theorem}

We introduce an additional equivalent characterization, complementary to those  in Theorem \ref{thm-characterizations-p-OAP}, regarding when a space possesses the $p$-OAP.  This one is related with approximation in the operator weakly $p$-compact norm, and establishes an operator space analogue of \cite[Thm. 6.3]{Sinha-Karn}.

\begin{theorem}\label{thm-p-OAP-and-weakly-p-compact}
Let $1 \le p \le \infty$.
An operator space $V$ has the $p$-OAP if and only if for every operator space $W$ we have that $\mathcal{F}(W,V)$ is $\omega_p^o$-dense in $\mathcal{K}_p^o(W,V)$.
\end{theorem}

\begin{proof}
Suppose that $V$ has the $p$-OAP. Let $W$ be an operator space, $T \in \mathcal{K}_p^o(W,V)$ and $\varepsilon>0$. Then, there exists $v \in \mathcal{S}_p[V]$ such that $T(\mathbf{B}_W) \subseteq \Theta^v(\mathbf{B}_{\mathcal{S}_{p'}})$. Since $V$ has the $p$-OAP, there exists $R \in \mathcal{F}(V,V)$ such that $\n{u}_{\mathcal{S}_p^w[V]} < \varepsilon$, where $u:=(Id_{\mathcal{S}_p}\otimes R)(v) - v$.
Let $S = R T \in \mathcal{F}(W,V)$.
Then
\[
(S - T)(\mathbf{B}_W)  = (R-Id_V)T(\mathbf{B}_W) \subseteq (R-Id_V) \Theta^v(\mathbf{B}_{\mathcal{S}_{p'}}) = \Theta^{u}(\mathbf{B}_{\mathcal{S}_{p'}}),
\]
which shows that $\omega_p^o(S - T) < \varepsilon$.

On the other hand, if $\mathcal{F}(W,V)$ is $\omega_p^o$-dense in $\mathcal{K}_p^o(W,V)$ then $\mathcal{F}(W,V)$ is $\n{\cdot}_{\cb}$-dense in $\mathcal{K}_p^o(W,V)$ (because $\n{\cdot}_{\cb} \le \omega_p^o(\cdot)$) so $V$ has the $p$-OAP by Theorem \ref{thm-characterizations-p-OAP}.
\end{proof}

\subsection{An application to approximation of Herz-Schur multipliers by finitely supported ones}

We now recall some basic terminology from abstract harmonic analysis. The reader is referred to \cite{Kaniuth-Lau} for more detailed background.
Given a discrete group $G$,
we let $\lambda : G \to \B(\ell_2(G))$ be the \emph{left regular representation}, that is, $\lambda(s)\xi(t) = \xi(s^{-1}t)$ for all $\xi \in \ell_2(G)$ and $s,t \in G$.
In particular, if $\{\delta_t\}_{t\in G}$ is the canonical basis of $\ell_2(G)$, then $\lambda(s) \delta_t = \delta_{st}$.
The \emph{reduced group $C^*$-algebra} $C^*_{\lambda}(G)$ and the \emph{group von  Neumann algebra} $\vN(G)$ are the norm closure and the weak$^*$ closure, respectively, of $\spa\{\lambda(t)\}_{t\in G}$ in $\B(\ell_2(G))$, which we will consider with their standard operator space structures.
The \emph{Fourier algebra} of $G$, denoted by $A(G)$, is the collection of scalar-valued functions on $G$ of the form $s \mapsto \pair{\lambda(s)\xi}{\eta}$ where $\xi, \eta \in \ell_2(G)$.
The Fourier algebra will be equipped with the canonical operator space structure arising with its identification as the predual of $\vN(G)$.

A function $\varphi : G \to \C$ is called a \emph{multiplier} of $A(G)$ if $f \mapsto \varphi f$ maps $A(G)$ into $A(G)$.
A multiplier $\varphi$ is called a \emph{Herz-Schur multiplier} if the map $f \mapsto \varphi f$ is completely bounded from $A(G)$ to $A(G)$.
This is equivalent to the complete boundedness of the linear map $m_\varphi : C^*_{\lambda}(G) \to C^*_{\lambda}(G)$ given by $m_\varphi( \lambda(t) ) = \varphi(t) \lambda(t)$, which can be shown in a manner analogous to that for the equivalence of the boundedness of the multiplier $\varphi$ and of $m_\varphi$ (see \cite[Rmk. 5.1.3]{Kaniuth-Lau}).

In the context of multipliers, it is known that approximation properties can yield refined versions of the ``approximating a compact map by finite-rank ones'' philosophy: in \cite[Cor. 3.9]{He-Todorov-Turowska} it is shown that if $C^*_{\lambda}(G)$ has SOAP, then a completely compact $m_\varphi$ can be approximated in $\cb$-norm using finitely supported multipliers.
Below we show an analogous result for operator $p$-compact multipliers in the presence of $p$-OAP, but first we isolate a standard result about going from finite-rank approximations to approximations by finitely supported multipliers. While the strategy is well-known to experts, we were not able to find the result explicitly stated in the literature and include its proof for completeness.

\begin{lemma}\label{lemma-from-finite-rank-to-finitely-supported}
Let $G$ be a discrete group, $\psi$ a Herz-Schur multiplier on $G$, and $\varepsilon>0$.
If there exists $T \in \mathcal{F}\big( C^*_\lambda(G), C^*_\lambda(G) \big)$ such that $\n{m_\psi - T}_{\cb} < \varepsilon$, then there exists a finitely supported multiplier $\varphi$ on $G$ such that  $\n{m_\psi - m_\varphi}_{\cb} < \varepsilon \n{m_{\psi}}_{\cb}$.   
\end{lemma}

\begin{proof}
Since $\spa\{ \lambda(s) \;:\; s \in G \}$ is norm dense in $C^*_\lambda(G)$, we can assume that $T$ takes values in $\spa\{ \lambda(s) \;:\; s \in F \}$ for some finite set $F \subset G$.
Define a function $\varphi : G \to \C$ by $\varphi(s) = \psi(s) \tau\big( \lambda(s)^* T(\lambda(s)) \big)$, where $\tau$ is the standard faithful normal tracial state on $\mathcal{B}(\ell_2(G))$ given by $\tau(a) = \pair{a \delta_e}{\delta_e}$. Note that $\varphi(s) = 0$ for any $s \in G \setminus F$, so $\varphi$ is a finitely supported function. A calculation shows that for any $a \in C^*_\lambda(G)$ we have
\[
m_\varphi(a) = J^*(m_\psi \otimes T)\pi(a)J \qquad\text{and}\qquad a = J^*\pi(a)J
\]
where $J : \ell_2(G) \to \ell_2(G) \widehat\otimes_2 \ell_2(G)$ is the isometry given by $J\delta_s = \delta_s \otimes \delta_s$ for every $s \in G$, and $\pi : C^*_\lambda(G) \to C^*_\lambda(G) \otimes C^*_\lambda(G)$ is the $*$-homomorphism given by $\pi(\lambda(s)) = \lambda(s) \otimes \lambda(s)$.
We remark that this is similar to the proofs of \cite[Thms. 12.2.15 and 12.3.10]{Brown-Ozawa}, where essentially the same strategy is used when the multiplier being approximated is the identity.
It follows that for any $a \in C^*_\lambda(G)$ we have
\[
m_\varphi(a) - m_\psi(a) = J^*(m_\psi \otimes (T-m_\psi) ) \pi(a) J,
\]
from where it is clear that $\n{m_\varphi - m_\psi}_{\cb} < \varepsilon \n{m_{\psi}}_{\cb}$.
\end{proof}

From Theorem \ref{thm-characterizations-p-OAP} and Lemma \ref{lemma-from-finite-rank-to-finitely-supported} we immediately get the announced result.

\begin{corollary}\label{cor-multipliers}
Let $G$ be a discrete group, and suppose that $C^*_\lambda(G)$ has $p$-OAP.
If $\psi$ is a Herz-Schur multiplier on $G$ such that $m_\psi$ is operator $p$-compact, then $m_\psi$ can be approximated in $\cb$-norm by finitely supported multipliers.
\end{corollary}

At first sight it may look like the case $p=\infty$ of our result above is not directly comparable with \cite[Cor. 3.9]{He-Todorov-Turowska}, which has a stronger assumption (SOAP vs. OAP) and a stronger conclusion (approximating completely compact multipliers vs. operator compact ones).
However, since $C^*_\lambda(G)$ has the OAP if and only if it has the SOAP \cite[Thm. 12.4.9]{Brown-Ozawa}, in reality \cite[Cor. 3.9]{He-Todorov-Turowska} is strictly stronger than the case $p=\infty$ of Corollary \ref{cor-multipliers}.

\begin{remark}
Note that by \cite[Prop. 3.2]{He-Todorov-Turowska} studying operator $p$-compact multipliers is only of interest in the case of discrete groups, because on a non-discrete locally compact group the only compact Herz-Schur multiplier is the zero map.    
\end{remark}

\section{The $p$-slice mapping property}

Both the OAP and its strong version can be understood in terms of \emph{slice mapping properties} \cite[Cor. 11.3.2]{Effros-Ruan-book}, a notion due to Tomiyama \cite{Tomiyama} with important applications in the theory of operator algebras.
Yew's $Z$-AP with respect to an operator space $Z$ can also be characterized in such terms \cite[Thm. 6.2(1)]{Yew}, and now we proceed to do the analogous study for the $p$-OAP.
We say that an operator space $V$ has the \emph{$p$-slice mapping property} if for any closed subspace $W \subseteq \mathcal{S}_p$ and any $u \in \mathcal{S}_p[V]$ satisfying that $(Id_{\mathcal{S}_p}\otimes v')(u) \in W$ for all $v'\in V'$, we have $u \in W\widehat{\otimes}_{\min} V$.

\begin{theorem}\label{thm-slice-mapping}
An operator space $V$ has the $p$-OAP if and only if it has the $p$-slice mapping property.    
\end{theorem}

\begin{proof}
This proof is analogous to that of \cite[Thm. 11.3.1]{Effros-Ruan-book}, but we include it for completeness.

Suppose that $V$ has $p$-OAP.
Consider a closed subspace $W \subseteq \mathcal{S}_p$ and  $u \in \mathcal{S}_p[V]$ satisfying that $(Id_{\mathcal{S}_p}\otimes v')(u) \in W$ for all $v'\in V'$.
By assumption, there exists $T\in\mathcal{F}(V,V)$ such that
\[
 \|(Id_{\mathcal{S}_p} \otimes T)(u) - u \|_{\mathcal{S}_p\widehat\otimes_{\min} V} < \varepsilon.
\]
Let us write $T = \sum_{j=1}^m v_j' \otimes v_j$ for some $v_j \in V$, $v_j' \in V'$.
Note that for any $x\in \mathcal{S}_p$ and $v \in V$ we have
\[
(Id_{\mathcal{S}_p} \otimes T)(x \otimes v) = \sum_{j=1}^n x \otimes v_j'(v)v_j = \sum_{j=1}^m (Id_{\mathcal{S}_p} \otimes v'_j)(x \otimes v) \otimes v_j.
\]
By linearity we get that for any $\overline{u} \in \mathcal{S}_p \otimes V$ we have
\[
(Id_{\mathcal{S}_p} \otimes T)(\overline{u}) = \sum_{j=1}^m (Id_{\mathcal{S}_p} \otimes v'_j)(\overline{u}) \otimes v_j,
\]
and taking limits we conclude that the same is true for all $\overline{u} \in \mathcal{S}_p\widehat\otimes_{\min}V$.
In particular, $u$ can be arbitrarily approximated in the $\mathcal{S}_p\widehat\otimes_{\min} V$ norm by elements of $W\otimes V$, which implies that $u \in W\widehat{\otimes}_{\min} V$.

Let us now assume that $V$ has the $p$-slice mapping property.
Let $u \in \mathcal{S}_p[V]$ be given. Let
\[
W_u = \overline{ \spa \{ (Id_{\mathcal{S}_p}\otimes v')(u) \;:\; v' \in V' \}  } \subseteq \mathcal{S}_p.
\]
By our assumption, $u \in W_u \widehat\otimes_{\min}V$.
Therefore, given $\varepsilon>0$ there exist $v_j' \in V'$, $v_j \in V$, $j=1,\dotsc, m$, such that
\[
u_\varepsilon = \sum_{j=1}^m (Id_{\mathcal{S}_p}\otimes v_j')(u) \otimes v_j
\]
satisfies $\n{u - u_\varepsilon}_{W_u\widehat\otimes_{\min}V} < \varepsilon$.
If we define $T = \sum_{j=1}^m v'_j \otimes v_j \in \mathcal{F}(V,V)$, then we will have
\[
\|(Id_{\mathcal{S}_p} \otimes T)(u) - u \|_{\mathcal{S}_p\widehat\otimes_{\min} V} = \n{u - u_\varepsilon}_{\mathcal{S}_p\widehat\otimes_{\min}V} = \n{u - u_\varepsilon}_{W_u\widehat\otimes_{\min}V} < \varepsilon
\]
\end{proof}

Just as in the case of the classical slice mapping property, the $p$-slice mapping property has a close relationship to a variation of the relative Fubini product.
For a closed subspace $W \subseteq \mathcal{S}_p$, its associated \emph{relative Fubini product}, as defined in \cite[Sec. 11.3.]{Effros-Ruan-book}, is
\[
\mathscr{F}(W,V) = \big\{ u\in \mathcal{S}_p \widehat\otimes_{\min} V \;:\; (Id_{\mathcal{S}_p}\otimes v')(u) \in W \text{ for all } v'\in V'\big\}.
\]

We define a \emph{relative $p$-Fubini product} as $\mathscr{F}_p(W,V) = \mathscr{F}(W,V)\,\cap\, \mathcal{S}_p[V]$, where we understand $\mathcal{S}_p[V]$ as a subset of $\mathcal{S}_p \widehat\otimes_{\min} V$ by Lemma \ref{lemma-Theta-well-defined}. Note that an operator space $V$ has the $p$-slice mapping mapping property (and therefore, equivalently, the $p$-OAP) if and only if for every closed subspace $W \subseteq \mathcal{S}_p$ we have $\mathscr{F}_p(W,V) \subseteq W\widehat{\otimes}_{\min} V$.

This approach will enable us to prove that what we demonstrated in Corollary \ref{cor-2-OAP} is a particular case of Theorem \ref{th-homogeneous-hilbertian}, which shows that every operator space has the $H$-AP (in the sense of Yew), whenever $H$ is a homogeneous Hilbertian operator space.

We begin by seeing, in the following proposition, an analogue to \cite[Prop. 11.3.4]{Effros-Ruan-book} which identifies the relative $p$-Fubini product as the kernel of a natural map.

\begin{proposition}\label{prop-Fubini-as-kernel}
Let $V$ be an operator space and $1 \le p \le \infty$.
If $W \subseteq \mathcal{S}_p$ is a closed subspace and $Q : \mathcal{S}_p \to \mathcal{S}_p/W$ is the canonical quotient, then the kernel of the mapping
$Q \otimes Id_V : \mathcal{S}_p[V] \to (\mathcal{S}_p/W) \widehat\otimes_{\min} V$
is $\mathscr{F}_p(W,V)$.
\end{proposition}

\begin{proof}
This is immediate from \cite[Prop. 11.3.4]{Effros-Ruan-book}, which says that the kernel of $Q \otimes Id_V : \mathcal{S}_p \widehat\otimes_{\min} V \to (\mathcal{S}_p/W) \widehat\otimes_{\min} V$ is $\mathscr{F}(W,V)$, simply by restricting to  $\mathcal{S}_p[V] \subset \mathcal{S}_p \widehat\otimes_{\min} V$.
\end{proof}

As is well-known to experts in $C^*$-algebra theory, this circle of ideas can be expressed nicely from the homological algebra point of view. 
The following is a restatement of \cite[Thm. 6.2(1)]{Yew} using \cite[Prop. 11.3.4]{Effros-Ruan-book}. We point out that this equivalence is reminiscent of characterizations of exactness for operator spaces involving exact sequences, such as \cite[Thms. 14.4.1 and 14.4.2]{Effros-Ruan-book}.

\begin{corollary}\label{cor-Z-AP-with-exact-sequences}
Let $Z$ and $V$ be operator spaces. $V$ has the $Z$-AP if and only if 
for every closed subspace $W \subseteq Z$, the short exact sequence
\[
0 \to W \to Z \to Z/W \to 0
\]
induces a short exact sequence
\begin{equation}\label{eqn-Z-AP-short-exact-sequence}
0 \to W\widehat{\otimes}_{\min} V \to Z\widehat{\otimes}_{\min} V \to (Z/W) \widehat\otimes_{\min} V \to 0. 
\end{equation}
\end{corollary}

\begin{proof}
Note the $Z$-AP for $V$ is equivalent to having $\mathscr{F}(W,V) \subseteq  W\widehat{\otimes}_{\min} V$ for all closed subspaces $W \subseteq Z$, as shown in \cite[Thm. 6.2(1)]{Yew}.
Since $W\widehat{\otimes}_{\min} V \subseteq \mathscr{F}(W,V)$ always holds, the $Z$-AP for $V$ is thus equivalent to having $\mathscr{F}(W,V) =  W\widehat{\otimes}_{\min} V$ for all closed subspaces $W \subseteq Z$. As used previously, we know from \cite[Prop. 11.3.4]{Effros-Ruan-book} that the kernel of $Z \widehat\otimes_{\min} V \to (Z/W) \widehat\otimes_{\min} V$ is $\mathscr{F}(W,V)$. This  precisely says that \eqref{eqn-Z-AP-short-exact-sequence} is an exact sequence.
\end{proof}

Note that in particular Corollary \ref{cor-Z-AP-with-exact-sequences} covers the case of the OAP which coincides with the $\mathcal{S}_\infty$-AP. For $1 \le p < \infty$, one has a corresponding result with an analogous proof:

\begin{corollary}\label{cor-p-OAP-with-exact-sequences}
Let $1 \le p \le \infty$. An operator space $V$ has $p$-OAP if and only if, for every closed subspace $W \subseteq \mathcal{S}_p$ the short exact sequence
\[
0 \to W \to \mathcal{S}_p \to \mathcal{S}_p/W \to 0
\]
induces a short exact sequence
\begin{equation}\label{eqn-p-OAP-short-exact-sequence}
0 \to \big(W\widehat{\otimes}_{\min} V\big) \cap \mathcal{S}_p[V] \to \mathcal{S}_p[V] \to (\mathcal{S}_p/W) \widehat\otimes_{\min} V \to 0,
\end{equation} 
where both $W\widehat{\otimes}_{\min} V$ and $S_p[V]$ are understood as subsets of $\mathcal{S}_p\widehat{\otimes}_{\min} V$ (and that is where the intersection is taken).
\end{corollary}

\begin{proof}
By Proposition \ref{prop-Fubini-as-kernel}, the kernel of the mapping $Q \otimes Id_V : \mathcal{S}_p[V] \to (\mathcal{S}_p/W) \widehat\otimes_{\min} V$ is $\mathscr{F}_p(W,V)$, where $Q : \mathcal{S}_p \to \mathcal{S}_p/W$ is the canonical quotient. Therefore, \eqref{eqn-p-OAP-short-exact-sequence} is exact if and only if $\big(W\widehat{\otimes}_{\min} V\big) \cap \mathcal{S}_p[V] = \mathscr{F}_p(W,V)$. Trivially we always have $W\widehat{\otimes}_{\min} V \subseteq \mathscr{F}(W,V)$, so intersecting with $\mathcal{S}_p[V]$ we always have  $\big(W\widehat{\otimes}_{\min} V\big) \cap \mathcal{S}_p[V] \subseteq \mathscr{F}_p(W,V)$.
Therefore, \eqref{eqn-p-OAP-short-exact-sequence} is exact if and only if $\mathscr{F}_p(W,V) \subseteq W\widehat{\otimes}_{\min} V$.
As pointed out above, having this condition for all closed subspaces $W \subseteq \mathcal{S}_p$ precisely characterizes that $V$ has the $p$-OAP.    
\end{proof}

\begin{remark}
We thank the anonymous referee for correcting a previous statement of ours regarding the completeness of the spaces appearing in \eqref{eqn-p-OAP-short-exact-sequence}.
Obviously $S_p[V]$ and $(\mathcal{S}_p/W) \widehat\otimes_{\min} V$ are complete spaces, but what about $\big(W\widehat{\otimes}_{\min} V\big) \cap \mathcal{S}_p[V]$?
First note that by Lemma \ref{lemma-Theta-well-defined} and the metric mapping property of the minimal tensor product, the map $Q \otimes Id_V : \mathcal{S}_p[V] \to (\mathcal{S}_p/W) \widehat\otimes_{\min} V$ is continuous and therefore its kernel (namely $\mathscr{F}_p(W,V)$) is closed in $\mathcal{S}_p[V]$.
Now, the space $\big(W\widehat{\otimes}_{\min} V\big) \cap \mathcal{S}_p[V]$ is the inverse image of the closed subspace $W\widehat{\otimes}_{\min} V \subseteq \mathcal{S}_p\widehat\otimes_{\min}V$ under the continuous injection $\mathcal{S}_p[V] \to \mathcal{S}_p\widehat\otimes_{\min}V$  of Lemma \ref{lemma-Theta-well-defined}, which implies that $\big(W\widehat{\otimes}_{\min} V\big) \cap \mathcal{S}_p[V]$ is closed in $ \mathcal{S}_p[V]$ (and thus complete).
\end{remark}

While the previous corollary may not be too satisfactory, we will now use this approach to prove a generalization of Corollary \ref{cor-2-OAP} which we believe greatly clarifies why it is that every operator space has the 2-OAP. We start with a lemma that says that the condition in Corollary \ref{cor-Z-AP-with-exact-sequences} is automatically satisfied for completely complemented subspaces.

\begin{lemma}\label{lemma-complemented-implies-exact}
Let $W \subseteq Z$ be a completely complemented subspace of the operator space $Z$. Then the sequence \eqref{eqn-Z-AP-short-exact-sequence} is exact.
\end{lemma}

\begin{proof}
Let $P : Z \to Z$ be a completely bounded projection onto $W$, and let $W_0 = (Id_Z - P)Z \subset Z$ be the corresponding complement. Note that in this case the sequence   \eqref{eqn-Z-AP-short-exact-sequence} is exact if and only if so is
\[
0 \to W\widehat{\otimes}_{\min} V \to Z\widehat{\otimes}_{\min} V \to W_0 \widehat\otimes_{\min} V \to 0.
\]
Now, from the mapping property of the minimal tensor product,
$P \otimes Id_V$ and $(Id_Z-P)\otimes Id_V$ are complementary projections in $\CB(Z\widehat{\otimes}_{\min} V,Z\widehat{\otimes}_{\min} V)$.
Moreover, their ranges are precisely $ W\widehat{\otimes}_{\min} V$ and $W_0\widehat{\otimes}_{\min} V$, respectively. Indeed, $Z \otimes V$ is dense in $Z\widehat{\otimes}_{\min} V$, so $P(Z) \otimes V = W \otimes V$ is dense in the range of $P \otimes Id_V$, which implies that the range of $P \otimes Id_V$ is precisely $ W\widehat{\otimes}_{\min} V$ (and analogously for the complementary projection).
Since the kernel of a projection always coincides with the range of the complementary projection, we are done.
\end{proof}

Recall that an operator space $H$ is called \emph{Hilbertian} if as a Banach space it is a Hilbert space, and \emph{homogeneous} if every bounded linear map $T : H \to H$ is completely bounded. See \cite[Sec. 9.2]{Pisier-OS-theory} for properties and examples of  homogeneous Hilbertian operator spaces.
In particular, $\mathcal{S}_2$ is homogeneous since it is completely isometrically isomorphic to $\OH$ (as pointed out in the proof of Corollary \ref{cor-2-OAP} above).

\begin{theorem}\label{th-homogeneous-hilbertian}
If $H$ is a homogeneous Hilbertian operator space, then every operator space has the $H$-AP.
In particular, every operator space has the $\mathcal{S}_2$-AP and therefore the $2$-OAP.
\end{theorem}

\begin{proof}
Any closed subspace $Z \subseteq H$ is complemented because $H$ is Hilbertian, and thus completely complemented by homogeneity. The desired conclusion now follows from Corollary \ref{cor-Z-AP-with-exact-sequences} and Lemma \ref{lemma-complemented-implies-exact}.    
\end{proof}

\section{Transference of the $p$-OAP from $V'$ to $V$} \label{sec transference}


The classical approximation property is well-known to transfer from the dual to the space, and the same is true for the $p$-AP
\cite[Thm. 2.7]{Choi-Kim}. In the context of operator spaces a similar transference result holds for the OAP \cite[Cor. 11.2.6]{Effros-Ruan-book}. 
In this section, we explore this type of relation for the $p$-OAP and demonstrate that the statement is valid with an additional hypothesis. To present this, we introduce the following definition, inspired by an analogous connection in the classical case which characterizes the Radon-Nikodým property for dual spaces, see \cite[Prop. 16.5]{Defant-Floret}.




\begin{definition}\label{defn-p-ORNP}
Let $1<p \le \infty$.
An operator space $V$ is said to have the \emph{\ORNP{p}} if the natural map $\mathcal{S}_{p'} \widehat\otimes_{\proj} V' \to (\mathcal{S}_p \widehat\otimes_{\min}V)'$ is surjective.
\end{definition}

Let us recall how the aforementioned natural map is defined, since this will be useful for us later in this section.
For any operator space $W$, by the universal property of the projective tensor product, the duality pairing $W \times W' \to \C$ corresponds to a complete contraction $W \widehat\otimes_{\proj} W' \to \C$ which is usually referred to as the \emph{tensor contraction}.
We will be particularly interested in the duality pairing between $\mathcal{S}_{p}$ and $\mathcal{S}_{p'}$, which is given by $\pair{a}{b} = \tr(ab^t)=\sum_{ij} a_{ij}b_{ij}$ where $a = (a_{ij}) \in \mathcal{S}_{p}$ and $b = (b_{ij}) \in \mathcal{S}_{p'}$.
For any operator spaces $V$ and $W$, it is known that this duality pairing induces a natural complete contraction (also called a tensor contraction)
\begin{equation}\label{eqn-Sp-tensor-contraction-general}
    (\mathcal{S}_p \widehat\otimes_{\min}V) \widehat\otimes_{\proj} (\mathcal{S}_{p'} \widehat\otimes_{\proj} W) \to V \widehat\otimes_{\proj} W
\end{equation}
which is given on elementary tensors by $a \otimes v \otimes b \otimes w \mapsto \pair{a}{b}v \otimes w$ \cite[Lemma 3.2]{CD-Chevet-Saphar}.
Note that in the special case $W=V'$, the mapping \eqref{eqn-Sp-tensor-contraction-general} can in turn then be composed with the tensor contraction $V \widehat\otimes_{\proj} V' \to \C$,
yielding a completely bounded linear map $ (\mathcal{S}_p \widehat\otimes_{\min}V) \widehat\otimes_{\proj} (\mathcal{S}_{p'} \widehat\otimes_{\proj} V') \to \C$.
By the universal property of the projective tensor product again \cite[Prop. 7.1.2]{Effros-Ruan-book}, this corresponds to the jointly completely bounded bilinear map $(\mathcal{S}_p \widehat\otimes_{\min}V) \times (\mathcal{S}_{p'} \widehat\otimes_{\proj} V') \to \C$ which is given on elementary tensors by $(a\otimes v, b\otimes v') \mapsto \pair{a}{b} \pair{v}{v'}$.
The natural map $\mathcal{S}_{p'} \widehat\otimes_{\proj} V' \to (\mathcal{S}_p \widehat\otimes_{\min}V)'$ appearing in Definition \ref{defn-p-ORNP} is precisely the one corresponding to said bilinear map.

We now exhibit examples of operator spaces with the \ORNP{p}. 
Recall that for a given Banach space $X$, $\MIN(X)$ (resp. $\MAX(X)$) is the unique operator structure on $X$ compatible with the norm of $X$ and such that for any operator space $V$ and any linear map $T : V \to X$ (resp. $S : X \to V$) we have
 $\|T : V \to \MIN(X)\|_{\cb} = \| T : V \to X \|$
 (resp. $\|S : \MAX(X) \to V\|_{\cb} = \| S : X \to V \|$).
 The reader is referred to \cite[Chap. 3]{Pisier-OS-theory} for more details.

\begin{proposition}\label{proposition-examples-p-ORNP}
\begin{enumerate}[(a)]
\item Let $1<p \leq \infty$ and let $X$ be a Banach space. Then $\MIN(X)$ has the \ORNP{p}. 
\item Every operator space has the \ORNP{\infty}.
\end{enumerate}
\end{proposition}

\begin{proof}
$(a)$ Note that $\MIN(X)'=\MAX(X')$ \cite[Sec. 1.4.12]{Blecher-LeMerdy}, and
$\mathcal{S}_{p'} \widehat\otimes_{\proj} \MAX(X') = \mathcal{S}_{p'} \widehat\otimes_{\pi} X'$ 
and $\mathcal{S}_p \widehat\otimes_{\min} \MIN(X) = \mathcal{S}_p \widehat\otimes_{\varepsilon} X$ as Banach spaces \cite[Prop. 1.5.3 and Prop. 1.5.12]{Blecher-LeMerdy}. Therefore the result follows from the classical characterization of the Radon-Nikod\'ym property for dual spaces \cite[Prop. 16.5]{Defant-Floret} (since $\mathcal{S}_{p'}$ is a separable dual space, and thus it has the Radon-Nikod\'ym property).

$(b)$ Is immediate because in this case we even have a complete isometry $\mathcal{S}_1[V']= \mathcal{S}_1\widehat\otimes_{\proj} V' \to (\mathcal{S}_\infty \widehat\otimes_{\min}V)' = \mathcal{S}_\infty[V]'$ \cite[Cor. 1.8]{Pisier-Asterisque}.


\end{proof}

The \ORNP{p} is a convenient technical tool, but at the moment it is still very mysterious to the authors. Besides Proposition \ref{proposition-examples-p-ORNP}, we do not know further examples of spaces that do or do not have it.

We now describe the functionals on $\CB(V,W)$ which are $\tau_p$-continuous. This is a noncommutative version of \cite[Thm. 2.5]{Choi-Kim}, although the expressions for the functionals in that result are more explicit than ours. We have chosen this  abstract presentation to emphasize the role of the \ORNP{p}, and postpone the precise formulas until Proposition \ref{prop-double-tensor-contraction}.

\begin{proposition}\label{prop-dual-of-tau_p}
Let $V,W$ be operator spaces and $1<p\le \infty$.
\begin{enumerate}[(a)]
\item For any $v=(v_{ij}) \in \mathcal{S}_p[V]$,
and $\omega \in \mathcal{S}_{p'} \widehat\otimes_{\proj} W'$
the expression
$
\varphi(T) = \pair{\omega}{ (Tv_{ij}) }
$
defines a functional $\varphi\in(\CB(V,W),\tau_p)'$.
\item If $W$ has the \ORNP{p}, then all of the functionals in $(\CB(V,W),\tau_p)'$ can be represented as in part (a). 
\end{enumerate}
\end{proposition}

\begin{proof}
$(a)$ First, let us note that for any $v=(v_{ij}) \in \mathcal{S}_p[V]$ and $T \in \CB(V,W)$ we have $(Tv_{ij}) \in \mathcal{S}_p[W]$ \cite[Cor. 1.2]{Pisier-Asterisque}, so the expression for $\varphi$ above makes sense because $\mathcal{S}_p[W]$ is contained in $\mathcal{S}_p \widehat\otimes_{\min}W$ (Lemma \ref{lemma-Theta-well-defined}) and because of the natural contractive map $\mathcal{S}_{p'} \widehat\otimes_{\proj} W' \to (\mathcal{S}_p \widehat\otimes_{\min}W)'$.

 Thus, if $\varphi : \CB(V,W) \to \C$ is given by $\varphi(T) = \pair{\omega}{ (Tv_{ij}) }$,
 then (see \eqref{eqn-equality-norm-K-norm-Sp-min})
 \[
 |\varphi(T)| \le \n{\omega} \n{(Tv_{ij})}_{\mathcal{S}_p\widehat\otimes_{\min}W} = \n{\omega} \n{T}_{\mathbf{K}}
 \] 
 where $\mathbf{K} = \Theta^v(\mathbf B_{\mathcal{S}_p'})$, showing that $\varphi : (\CB(V,W),\tau_p) \to \C$ is continuous because of Proposition \ref{prop-uniform-convergence-over-p-compact}.

$(b)$ Now suppose $\varphi : (\CB(V,W),\tau_p) \to \C$ is continuous.
Then, there exist a constant $C$ and $v = (v_{ij}) \in \mathcal{S}_p[V]$ such that for every $T\in\CB(V,W)$ we have $|\varphi(T)| \le C \n{T}_{\mathbf{K}}$ where $\mathbf{K} = \Theta^v(\mathbf B_{\mathcal{S}_p'})$, that is, $|\varphi(T)| \le C \n{(Tv_{ij})}_{\mathcal{S}_p\widehat\otimes_{\min} W}$.
Therefore, there is a linear functional $\omega$ defined on the closure of the image of the mapping $T \in \CB(V,W) \mapsto (Tv_{ij}) \in \mathcal{S}_p\widehat\otimes_{\min} W$ satisfying $\varphi(T) = \omega((Tv_{ij}))$, for every $T \in \CB(V,W)$. By Hahn-Banach we can extend $\omega$ to a functional in  $(\mathcal{S}_p\widehat\otimes_{\min} W)'$ that we still denote by $\omega$, which is given by some element of $\mathcal{S}_{p'} \widehat\otimes_{\proj} W'$ because $W$ has the \ORNP{p}.
\end{proof}

For the proof of the main result of this section, it will be useful to understand the previous proposition in a more conceptual way. Recall from \eqref{eqn-Sp-tensor-contraction-general} that we have a jointly completely contractive bilinear map (induced by the duality pairing/tensor contraction)
\begin{equation}\label{eqn-Sp-tensor-contraction-specialized}
(\mathcal{S}_p \widehat\otimes_{\min}V) \times (\mathcal{S}_{p'} \widehat\otimes_{\proj} W') \to V \widehat\otimes_{\proj} W' \subseteq \CB(V,W)',    
\end{equation}
where the last inclusion is given by restricting the duality given by $(V \widehat\otimes_{\proj} W')' =\CB(V,W'')$ \cite[Thm. 4.1]{Pisier-OS-theory} to $\CB(V,W) \subset \CB(V,W'')$.
What Proposition \ref{prop-dual-of-tau_p} says is that when in \eqref{eqn-Sp-tensor-contraction-specialized} we restrict to $S_p[V] \subset \mathcal{S}_p \widehat\otimes_{\min}V$ (Lemma \ref{lemma-Theta-well-defined}), we get a bilinear map
\[
\mathcal{S}_p[V] \times (\mathcal{S}_{p'} \widehat\otimes_{\proj} W') \to (\CB(V,W),\tau_p)'
\]
which is induced by the duality pairing/tensor contraction. Moreover, when $W$ has the \ORNP{p} then this bilinear map is a surjection.
We build upon this in the next result.
Note that part $(a)$ is an infinite version of \cite[Prop. 1.15]{Pisier-Asterisque}.


\begin{proposition}\label{prop-double-tensor-contraction}
Let $1<p \le \infty$, and let $V$, $W$ be operator spaces.
\begin{enumerate}[(a)]
\item The duality pairing $\mathcal{S}_p \times \mathcal{S}_{p'} \to \C$  induces a surjection $\mathscr{C}_p : \mathcal{S}_p[V] \times \mathcal{S}_{p'}[W] \to V \widehat\otimes_{\proj} W$,
which is given by $\mathscr{C}_p\big( (v_{ij}),(w_{ij})\big) = \sum_{i,j} v_{ij} \otimes w_{ij} =  \lim_{N \to \infty} \sum_{i,j=1}^N v_{ij} \otimes w_{ij}$ for $(v_{ij}) \in S_p[V]$ and $(w_{ij}) \in \mathcal{S}_{p'}[W]$.
Moreover, for any $u \in V \widehat\otimes_{\proj} W$ we have
\[
\n{u}_{\proj} = \inf\big\{ \n{v}_{S_p[V]} \n{w}_{\mathcal{S}_{p'}[W]} \;:\; u = \mathscr{C}_p(v,w)\big\}.
\]
\item The double duality pairing induces a trilinear map 
$\mathscr{D}_p : \mathcal{S}_p[V] \times \mathcal{S}_{p'}[\mathcal{S}_{p'}] \times \mathcal{S}_{p}[W'] \to (\CB(V,W),\tau_p)'$ given as follows: for $v=(v_{ij})_{ij} \in \mathcal{S}_p[V]$, $a=\big( (a^{ij}_{kl})_{kl} \big)_{ij} \in \mathcal{S}_{p'}[\mathcal{S}_{p'}]$, $w'=(w'_{kl}) \in \mathcal{S}_{p}[W']$ and $T \in \CB(V,W)$,
\[
\mathscr{D}_p(v,a,w')(T) = \sum_{ijkl} a^{ij}_{kl} \pair{w'_{kl}}{Tv_{ij}} = \lim_{N \to \infty} \sum_{ijkl=1}^N a^{ij}_{kl} \pair{w'_{kl}}{Tv_{ij}}.
\]
If $W$ has the \ORNP{p}, then the map $\mathscr{D}_p$ is surjective.
\end{enumerate}  
\end{proposition}

\begin{proof}
$(a)$
For finitely supported matrices $v \in \mathcal{S}_p[V]$ and $w \in \mathcal{S}_{p'}[W]$, \cite[Prop. 1.15]{Pisier-Asterisque} says precisely that $\n{ \mathscr{C}_p(u,v) } \le \n{v}_{S_p[V]} \n{w}_{\mathcal{S}_{p'}[W]}$.
Approximating arbitrary $v \in \mathcal{S}_p[V]$ and $w \in \mathcal{S}_{p'}[W]$ using finitely supported matrices, we obtain the same inequality (and also that the limit defining $\mathscr{C}_p$ exists).
Thus, for any $u \in V \widehat\otimes_{\proj} W$ we have that $\n{u}_{\proj}$ is less than or equal to the infimum in the statement.
Now, since $V \otimes W$ is dense in $V \widehat\otimes_{\proj} W$ we can write $u$ as an absolutely convergent series $\sum_j u_j$  of elements  $u_j \in V \otimes W$, and with $\sum_j \n{u_j}_{\proj}$ arbitrarily close to  $\n{u}_{\proj}$. 
For each $u_j$, use \cite[Prop. 1.15]{Pisier-Asterisque} to approximate $\n{u_j}_{\proj}$ by $\n{v_j}_{\mathcal{S}_p^{n_j}[V]} \n{w_j}_{\mathcal{S}_{p'}^{n_j}[W]}$ where $u_j = \mathscr{C}_p(v_j,w_j)$. By rescaling (that is, replacing $v_j$ and $w_j$ by $\alpha_j v_j$ and $\alpha_j^{-1} w_j$, respectively), we can assume (with the obvious modification in the case $p=\infty$)
\[
\sum_{j=1}^\infty \n{v_j}_{\mathcal{S}_p^{n_j}[V]} \n{w_j}_{\mathcal{S}_{p'}^{n_j}[W]} = \left( \sum_{j=1}^\infty \n{v_j}_{\mathcal{S}_p^{n_j}[V]}^p \right)^{1/p} \left( \sum_{j=1}^\infty \n{w_j}_{\mathcal{S}_{p'}^{n_j}[W]}^{p'} \right)^{1/p'}.
\]
Then by \cite[Cor. 1.3]{Pisier-Asterisque} the infinite block-diagonal matrix $v$ (resp. $w$) with the $v_j$'s (resp. the $w_j$'s) is in $\mathcal{S}_p[V]$ (resp. in $\mathcal{S}_{p'}[W]$) and $u = \mathscr{C}_p(v,w)$. Moreover, $\n{v}_{\mathcal{S}_p[V]} \n{w}_{\mathcal{S}_{p'}[W]}$ is arbitrarily close to $\sum_j \n{u_j}_{\proj}$ which in turn is arbitrarily close to $\n{u}_{\proj}$.


$(b)$ From the discussion preceding the statement of the proposition, we have a bilinear surjection
\begin{equation}\label{prop-double-tensor-contraction-eqn-1}
\mathscr{E} : \mathcal{S}_p[V] \times (\mathcal{S}_{p'} \widehat\otimes_{\proj} W') \to (\CB(V,W),\tau_p)'    
\end{equation}
induced by the tensor contraction,
and by part $(a)$ (more precisely, a transposed version of it) we have a jointly contractive bilinear surjection
\begin{equation}\label{prop-double-tensor-contraction-eqn-2}
\mathscr{C} : \mathcal{S}_{p'}[\mathcal{S}_{p'}] \times \mathcal{S}_{p}[W'] \to \mathcal{S}_{p'} \widehat\otimes_{\proj} W'
\end{equation}
which is also induced by the tensor contraction (note we have not used subindices for these bilinear maps to keep the notation simpler, but they certainly depend on $p$).
The aforementioned two maps then naturally define a trilinear surjective map
\[
\mathscr{D} : \mathcal{S}_p[V] \times \mathcal{S}_{p'}[\mathcal{S}_{p'}] \times \mathcal{S}_{p}[W'] \to (\CB(V,W),\tau_p)', 
\]
which we understand as a double tensor contraction, and is given by
\[
\mathscr{D}(v, a, w')(T) = \mathscr{E}(v,\mathscr{C}(a,w') )(T), \qquad T \in \CB(V,W)
\]
whenever $v=(v_{ij})_{ij} \in \mathcal{S}_p[V]$, $a=\big( (a^{ij}_{kl})_{kl} \big)_{ij} \in \mathcal{S}_{p'}[\mathcal{S}_{p'}]$, $w'=(w'_{kl}) \in \mathcal{S}_{p}[W']$.
Let us assume that $v$, $a$, and $w'$ are finitely supported, that is,
\[
v = \sum_{i,j=1}^N \varepsilon_{ij} \otimes v_{ij}, \quad a = \sum_{i,j,k,l=1}^N a^{ij}_{kl} \varepsilon_{ij}
 \otimes \varepsilon_{kl}, \quad w' = \sum_{k,l=1}^N \varepsilon_{kl} \otimes w'_{kl}\]
where $\varepsilon_{ij}$ are the matrix units. It is then clear that $\mathscr{C}(a,w') = \sum_{i,j=1}^N a^{ij}_{kl} \varepsilon_{ij} \otimes w'_{kl}$, and thus
\[
\mathscr{D}(v, a, w')(T) = \mathscr{E}(v,\mathscr{C}(a,w') )(T) = \sum_{ijkl=1}^N a^{ij}_{kl} \pair{w'_{kl}}{Tv_{ij}}, \qquad T \in \CB(V,W).
\]
For the general case, we just need to verify that things work well when we approximate $v$, $a$, and $w'$ by  their finitely supported blocks. This follows since \eqref{prop-double-tensor-contraction-eqn-2} is jointly contractive and \eqref{prop-double-tensor-contraction-eqn-1} is jointly (completely) contractive when considered as a map taking values in $V \widehat\otimes_{\proj} W' \subseteq \CB(V,W)'$, which means that $\mathscr{D}$ is jointly contractive as a trilinear map $\mathcal{S}_p[V] \times \mathcal{S}_{p'}[\mathcal{S}_{p'}] \times \mathcal{S}_{p}[W'] \to V \widehat\otimes_{\proj} W' \subseteq \CB(V,W)'$.

\end{proof}

As happened in Banach spaces (\cite[Lem. 2.6]{Choi-Kim}, \cite[Lem. 1.e.17]{Lindenstrauss-Tzafriri-I}), the adjoints of finite-rank mappings are $\tau_p$-dense in the space of finite-rank mappings on the dual. 

\begin{lemma}\label{lemma-Choi-Kim-2-6}
Let $V$ be an operator space, and define $\mathcal{F}'(V,V) = \{ T' \;:\; T\in\mathcal{F}(V,V)\} \subseteq \mathcal{F}(V',V')$. Then for each $1 \le p \le \infty$, $\mathcal{F}'(V,V)$ is $\tau_p$-dense in $\mathcal{F}(V',V')$.
\end{lemma}

\begin{proof}
Let $v'' \in V''$ and $v' \in V'$. By Goldstine's theorem, there is a bounded net $(v_\alpha)$ in $V$ which converges to $v''$ in the weak$^*$ topology.
Therefore, $v_\alpha \otimes v'$ is a bounded net in $\mathcal{F}'(V,V)$ which converges to $v'' \otimes v' \in \mathcal{F}(V,V)$ in the point-norm topology.
By Proposition \ref{proposition-bounded-net}, this implies $\tau_p$-convergence.
Since every mapping in $\mathcal{F}(V',V')$ is a finite sum of mappings of the form $v'' \otimes v'$, the conclusion follows.
\end{proof}

We are now in condition to prove a transference result of the $p$-OAP from the dual to the space under the $p$-tensor surjectivity property.
 Note that from Proposition \ref{proposition-examples-p-ORNP}, the case $p=\infty$ is precisely the known result that OAP always passes from the dual to the space without needing additional assumptions \cite[Cor. 11.2.6]{Effros-Ruan-book}.

\begin{theorem}\label{thm-transfer-p-OAP-from-dual}
Let $1<p\le \infty$, and suppose $V$ has the \ORNP{p}.
If $V'$ has the $p$-OAP, then so does $V$.   
\end{theorem}

\begin{proof}
It suffices to show that if $\varphi \in (\CB(V,V),\tau_p)'$ and $\varphi(T)=0$ for all $T \in \mathcal{F}(V,V)$, then $\varphi(Id_V) = 0$.

By Proposition \ref{prop-double-tensor-contraction}, since $V$ has the \ORNP{p} there exist
$v=(v_{ij})_{ij} \in \mathcal{S}_p[V]$, $a=\big( (a^{ij}_{kl})_{kl} \big)_{ij} \in \mathcal{S}_{p'}[\mathcal{S}_{p'}]$ and $v'=(v'_{kl}) \in \mathcal{S}_{p}[V']$ such that $\varphi=\mathscr{D}_p(v,a,v')$.
Using the Fubini theorem for Schatten spaces \cite[Thm. 1.9]{Pisier-Asterisque}, $a^t = \big( (a^{ij}_{kl})_{ij} \big)_{kl} \in \mathcal{S}_{p'}[\mathcal{S}_{p'}]$.
Since $\iota_V : V \to V''$ is a complete isometry so is $Id_{\mathcal{S}_p}\otimes \iota_V : \mathcal{S}_p[V] \to \mathcal{S}_p[V'']$ \cite[Cor. 1.2]{Pisier-Asterisque}. Then, $\iota_Vv=(\iota_V v_{ij}) \in \mathcal{S}_p[V'']$ and thus by Proposition \ref{prop-double-tensor-contraction} we can define a functional $\psi \in (\CB(V',V'),\tau_p)'$ as $\psi=\mathscr{D}_p( v', a^t, \iota_Vv )$.
Checking the formulas, one readily verifies that for all $T \in \CB(V,V)$ we have $\psi(T')=\varphi(T)$.
Therefore, by our assumption it follows that $\psi(T')=0$ for all $T\in \mathcal{F}(V,V)$.

Using that $V'$ has the $p$-OAP, we know from Proposition \ref{prop-p-OAP-density-in-CB} that $\mathcal{F}(V',V')$ is $\tau_p$-dense in $\CB(V',V')$.
By Lemma \ref{lemma-Choi-Kim-2-6}, we in fact have that $\mathcal{F}'(V,V)$ is $\tau_p$-dense in $\CB(V',V')$. Since $\psi$ vanishes on $\mathcal{F}'(V,V)$ and is $\tau_p$-continuous, $\psi$ must vanish on all of $\CB(V',V')$ and thus $0 = \psi(Id_{V'}) = \varphi(Id_V)$, which completes the proof.
\end{proof}

The AP for a Banach space $X$ is well-known to be equivalent to the following: whenever $(x_n)$ and $(x_n')$ are sequences in $X$ and $X'$, respectively, satisfying that $\sum_{n=1}^\infty \n{x_n}\n{x_n'}<\infty$ and $\sum_{n=1}^\infty \pair{x_n'}{x}x_n=0$ for all $x\in X$, we have $\sum_{n=1}^\infty \pair{x_n'}{x_n} = 0$ \cite[Thm. 1.e.4]{Lindenstrauss-Tzafriri-I}. This is typically called a ``trace condition'', because it means that the trace is well defined on the image of the canonical map $X \widehat\otimes_{\pi} X' \to X \widehat\otimes_{\varepsilon} X' \subseteq \mathcal{L}(X,X)$. Other trace conditions have been found to characterize the $p$-AP for Banach spaces \cite[Prop. 3.1]{Delgado-Oja-Pineiro-Serrano}, and also the OAP for operator spaces \cite[Thm. 11.2.5, condition (iv)]{Effros-Ruan-book}.
By adjusting the arguments of \cite[Prop. 3.1]{Delgado-Oja-Pineiro-Serrano}
to the operator space setting in the manner of \cite[Thm. 11.2.5]{Effros-Ruan-book}, Proposition \ref{prop-double-tensor-contraction} gives the trace condition below for the $p$-OAP (under the additional assumption of the \ORNP{p}).

\begin{corollary}\label{cor-trace-condition}
Let $1<p\le\infty$ and let $V$ be an operator space with the \ORNP{p}. The following are equivalent:
\begin{enumerate}[(i)]
\item $V$ has the $p$-OAP.
\item Whenever $(v_{ij})_{ij} \in \mathcal{S}_p[V]$, $\big( (a^{ij}_{kl})_{kl} \big)_{ij} \in \mathcal{S}_{p'}[\mathcal{S}_{p'}]$, and $(v'_{kl}) \in \mathcal{S}_{p}[V']$ satisfy that $\sum_{ijkl} a^{ij}_{kl} \pair{v'_{kl}}{u} v_{ij} = 0$ for all $u \in V$ , we have $\sum_{ijkl} a^{ij}_{kl} \pair{v'_{kl}}{v_{ij}} = 0$.
\end{enumerate}
\end{corollary}

\begin{proof}
 $(i) \Rightarrow (ii)$: Suppose that $V$ has the $p$-OAP, and let $v = (v_{ij})_{ij}$, $a = \big((a^{ij}_{kl})_{kl} \big)_{ij}$, and $v' = (v'_{kl})$ satisfy the assumption in part (ii).
 It follows that for every $w' \in V'$ and every $u \in V$ we have
 \[
 0 = w'(0) = w'\Big( \sum_{ijkl} a^{ij}_{kl} \pair{v'_{kl}}{u} v_{ij} \Big) = \sum_{ijkl} a^{ij}_{kl} \pair{v'_{kl}}{w'(v_{ij})u} = \mathscr{D}_p(v,a,v')(w' \otimes u),
 \]
 where $\mathscr{D}_p$ is the map defined in Proposition \ref{prop-double-tensor-contraction}.
 Therefore $\mathscr{D}_p(v,a,v')$ is a $\tau_p$-continuous functional on $\CB(V,V)$ that vanishes on $\mathcal{F}(V,V)$. Since $V$ has the $p$-OAP this implies  $\mathscr{D}_p(v,a,v')(Id_V)=0$, which is precisely the desired conclusion. 

$(ii) \Rightarrow (i)$:
Let $\varphi : \CB(V,V) \to \C$ be a $\tau_p$-continuous linear functional vanishing on $\mathcal{F}(V,V)$. Since $V$ has the \ORNP{p}, by Proposition \ref{prop-double-tensor-contraction} there exist $v=(v_{ij})_{ij} \in \mathcal{S}_p[V]$, $a=\big( (a^{ij}_{kl})_{kl} \big)_{ij} \in \mathcal{S}_{p'}[\mathcal{S}_{p'}]$, and $v'=(v'_{kl}) \in \mathcal{S}_{p}[V']$ such that for any $T \in \CB(V,V)$ we have $\varphi(T) = \mathscr{D}_p(v,a,v')(T)$.
Analogous calculations as in the first part of the proof show that the condition of $\varphi$ vanishing on $\mathcal{F}(V,V)$ implies that the assumption in condition (ii) is satisfied, which allows us to conclude $\sum_{ijkl} a^{ij}_{kl} \pair{v'_{kl}}{v_{ij}} = 0$, that is, $0 = \mathscr{D}_p(v,a,v')(Id_V) = \varphi(Id_V)$.
By Hahn-Banach we conclude $Id_V$ belongs to the $\tau_p$-closure of $\mathcal{F}(V,V)$, that is, $V$ has the $p$-OAP.
\end{proof}

We remark that the corollary above is indeed a trace condition, in the sense that part $(ii)$ is equivalent to having the trace being well defined on a certain class of mappings: those in the range of the composition of the canonical contractions below  
\[
\mathcal{S}_p[V] \times (\mathcal{S}_{p'} \widehat\otimes_{\proj} V') \to V \widehat\otimes_{\proj} V' \rightarrow V \widehat\otimes_{\min} V' \subseteq \CB(V,V).
\]
The corresponding class of mappings in the Banach space setting is described in \cite[Paragraph after Prop. 3.1]{Delgado-Oja-Pineiro-Serrano}.

We also remark that the implication (i) $\Rightarrow$ (ii) in Corollary \ref{cor-trace-condition} always holds, regardless of whether or not $V$ has the \ORNP{p}.

\newcommand{\etalchar}[1]{$^{#1}$}
\providecommand{\bysame}{\leavevmode\hbox to3em{\hrulefill}\thinspace}
\providecommand{\MR}{\relax\ifhmode\unskip\space\fi MR }
\providecommand{\MRhref}[2]{%
  \href{http://www.ams.org/mathscinet-getitem?mr=#1}{#2}
}
\providecommand{\href}[2]{#2}

\end{document}